\documentclass{article}

\usepackage{amsmath,amssymb,amsthm}
\usepackage[utf8]{inputenc}
\usepackage[T1]{fontenc}
\usepackage[a4paper]{geometry}
\usepackage{hyperref}
\hypersetup{bookmarksopen=true,colorlinks=true,allcolors={blue},}
\usepackage{authblk}
\usepackage{graphicx,caption,subcaption}

\usepackage{cancel}

\usepackage[
    sorting=nyt,
    minbibnames=5,
    maxbibnames=5,
    giveninits=true,
    url=false,
    isbn=false,
    eprint=true,
    doi=true
]{biblatex}
\theoremstyle{plain}
\newtheorem{theorem}{Theorem}[section]

\newtheorem{lemma}[theorem]{Lemma}
\newtheorem{corollary}[theorem]{Corollary}
\newtheorem{definition}[theorem]{Definition}

\newtheorem{remark}[theorem]{Remark}
\newtheorem{assumption}[theorem]{Assumption}

\numberwithin{equation}{section}
\author[,1]{Frédéric Nataf\thanks{{frederic.nataf@sorbonne-universite.fr}}}
\author[,1]{Emile Parolin\thanks{{emile.parolin@inria.fr}}}
\affil[1]{Sorbonne Université, Université Paris Cité, CNRS, INRIA, Laboratoire Jacques-Louis Lions, LJLL, EPC ALPINES, F-75005 Paris, France}

\title{
    Coarse spaces for non-symmetric two-level preconditioners
    based on local extended generalized eigenproblems
}

\date{\today}

\begin{document}
\maketitle
\begin{abstract}
    Domain decomposition (DD) methods are a natural way to take advantage of
    parallel computers when solving large scale linear systems. 
    Their scalability depends on the design of the coarse space used in the
    two-level method.  
    The analysis of adaptive coarse spaces we present here is quite general
    since it applies to symmetric and non-symmetric problems, to symmetric
    preconditioners such as the additive Schwarz method (ASM) and to the
    non-symmetric preconditioner restricted additive Schwarz (RAS), as well as
    to exact or inexact subdomain solves.
    The coarse space is built by solving generalized eigenproblems in the
    subdomains and applying a well-chosen operator to the selected
    eigenvectors.
\end{abstract}



\clearpage
\tableofcontents
\clearpage

\section{Introduction}

The scalability of domain decomposition methods depends on the design of a
suitable coarse
space~\cite{Quarteroni1999,Toselli:2005:DDM,Mathew:2008:DDM,Mandel:2011:CSA,Dolean:2015:IDDSiam}.
As with multigrid methods, an efficient coarse space should contain the
near-kernel of the matrix. For a Poisson problem, in~\cite{Nicolaides:DCG:1987}
the coarse space is built from constant functions that are multiplied for each
subdomain by the local value of the partition of unity function. For elasticity
problems, effective coarse spaces are made of rigid body modes also multiplied
by partition of unity functions, see~\cite{Mandel:1992:BDD}. When the
coefficients of the underlying partial differential equations are not smooth
per subdomain, the coarse space has to be enriched adaptively by selecting
eigenvectors obtained by solving local generalized eigenvalue problems (GEVP),
see
\cite{Mandel:2007:ASF,Nataf:CSC:2011,Efendiev:2012:RDD,Spillane:2014:ASC,Gander2017SHEM,Haferssas:ASM:2017,Bootland:2023:overlapping}
and references therein. For more algebraic works, we refer also to
\cite{AlDaas:2022:robust,AlDaas:2023:efficient,Heinlein:2022:fully} and
references therein.
For inexact local solvers in the subdomains, we refer to the analysis
in~\cite{Spillane2025}.
Note that in some works, the local GEVPs are restricted to
local harmonic functions see e.g.\ \cite{Nataf:CSC:2011} and the recent
works~\cite{Ma:2022:novel,Ma:2022:error,Hu:2024:novel,Ma:2024:TLR}. Actually,
in~\cite{Ma:2022:novel,Ma:2022:error,Ma:2024:TLR} such constructions are adapted to the
design of multiscale finite element methods where the coarse space is used as a
kind of reduced order space. There, the GenEO (Generalized Eigenvalue problem
in the Overlap) introduced in~\cite{Spillane:2011:RTL,Spillane:2014:ASC} is
restricted to local harmonic vectors/functions. Closer to our work
is~\cite{Hu:2024:novel} where the coarse space is analyzed in the context of
domain decomposition preconditioner. More precisely, the authors add a second
level to the ORAS~\cite{STCYR:2007:OMA}  (Optimized Restricted Additive
Schwarz) preconditioner which is commonly used to solve wave propagation
problems, see~\cite{Despres:1992:ADD} for the original one-level method. It is
proved that an efficient coarse space can be built by solving in each subdomain
a GEVP for locally harmonic functions.

In this work, we design and analyse
spectral coarse space for quite general domain decomposition methods such as
RAS (Restricted Additive Schwarz), AS (Additive Schwarz) or SORAS (Symmetric
Optimized Restricted Additive Schwarz) without assuming that the local solvers
yield errors that are harmonic in the subdomains. This includes inexact local
solves in the subdomains. Then the contribution of the local GEVP to the coarse
space is built from local eigenvectors that are non-harmonic.
We also show, under suitable assumptions, that when the local solvers yield
errors that are harmonic in the subdomains, we recover the coarse space
introduced in~\cite{Ma:2022:novel}.

In summary, the main highlights of our work are:
\begin{itemize}
    \item Design and analysis of a new coarse space named Extended GenEO (see
        Definition~\ref{def:newCS}) that is valid for symmetric and non-symmetric
        matrices and/or one-level preconditioners;
    \item A first convergence proof for a two-level RAS with inexact local
        solves (the alternative analysis of~\cite{Spillane2025} covers the case
        of symmetric preconditioners);
    \item With exact local solves, the local components of the coarse space are
        local harmonic functions multiplied by a partition of unity, see
        Lemma~\ref{lem:exactlocalsolvesHarmonic}. 
    \item When applied to SPD (Symmetric Definite Positive) problems with exact
        local solves, we prove in Lemma~\ref{th:harmonicgev} that the GEVP that
        yields the coarse space is equivalent to the ones
        in~\cite{Ma:2022:novel}.
    \item Numerical results for highly heterogeneous non-symmetric problems. 
\end{itemize}

In~\S~\ref{sec:abstract_analysis}, we present an abstract analysis of two-level
domain decomposition methods. In \S~\ref{sec:practical_methods}, this general
framework is applied to the RAS, AS and SORAS methods.
In~\S~\ref{sec:comparision_with_existing_CS}, we compare the extended GenEO coarse spaces
to those designed and analyzed using the fictitious space
lemma~\cite{Nepomnyaschikh:1991:MTT,Griebel:1995:ATA} in the symmetric positive
definite case (see~\cite{Spillane:2011:RTL,Dolean:2015:IDDSiam} for the AS
method and \cite{Haferssas:ASM:2017} for SORAS).

\section{Abstract analysis}
    \label{sec:abstract_analysis}
\subsection{Setting}

We consider the linear system
\begin{equation}
    \label{eq:syslin}
    A u = f,
\end{equation}
where \(A \in \mathbb{C}^{\#\mathcal{N}\times\#\mathcal{N}}\),
\(f \in \mathbb{C}^{\#\mathcal{N}}\),
\(u \in \mathbb{C}^{\#\mathcal{N}}\)
and \(\mathcal{N}\) is the set of unknowns with \(\#\mathcal{N}\) denoting
its cardinal.

The matrix \(A\) is assumed to be invertible so that the above problem is
well-posed, but it is not assumed a priori to be symmetric positive definite
(SPD).

\subsection{Domain decomposition preconditioners}

The global set of indices denoted by \(\mathcal{N}\) is decomposed, with or
without overlaps, into \(J\) subsets \((\mathcal{N}_j)_{= 1, \dots, J}\)
defining \(J\) subdomains. 

\paragraph{Restriction matrices}
The decomposition of the set of degrees of freedom \(\mathcal{N}\) is
characterized by the Boolean restriction matrices
\begin{equation}
    R_{j} \in \mathbb{R}^{\# {\mathcal{N}_j}\times \#\mathcal{N}},
    \qquad j = 1, \dots, J.
\end{equation}
We assume the following usual property for this decomposition
\begin{align}
    & R_{j}^{} R_{j}^{*} = I_j,
    \qquad j = 1, \dots, J,
\end{align}
where \(I_j \in \mathbb{R}^{\# {\mathcal{N}_j}\times \#\mathcal{N}_{j}}\) is the
identity operator on the subdomain \(j\).

\paragraph{Partition of unity}

This decomposition is moreover associated to a partition of unity given by
diagonal matrices with non-negative entries
\begin{equation}
    D_{j} \in \mathbb{R}^{\# {\mathcal{N}_j}\times \# \mathcal{N}_j},
    \qquad j = 1, \dots, J,
\end{equation}
such that the following identity is valid
\begin{equation}
    I = \sum_{j=1}^{J} R_{j}^{*} D_{j}^{} R_{j}^{},
\end{equation}
where \(I \in \mathbb{R}^{\# {\mathcal{N}}\times \#\mathcal{N}}\) is the
identity operator on the non-decomposed domain.

\paragraph{Local solvers}

In order to be able to define a preconditioner we introduce a family of
local solvers which are characterized by square matrices of size
\(\#\mathcal{N}_j \times \#\mathcal{N}_j\) not necessarily invertible:
\begin{equation}
    S_{j} \in {\mathbb C}^{\#\mathcal{N}_j \times \#\mathcal{N}_j},
    \qquad j = 1, \dots, J.
\end{equation}
This section is purposely general, so we do not assume a particular form for
the local solvers. 
Nevertheless, to fix ideas, let's give some standard examples.
Consider some invertible operator \(B_{j}\),
typically equal to the restriction of the global matrix to the subdomain \(j\): 
\(R_{j}^{} A R_{j}^{*}\) (for a SPD matrix \(A\)), or an approximation thereof.
Taking the local solver to be
\begin{equation}\label{eq:ras}
    S_{j}^{} := D_{j}^{} B_{j}^{-1},
    \qquad j=1,\dots,J,
\end{equation}
defines the Restricted Additive Schwarz (RAS)~\cite{Cai:1999:RAS} method in the
forthcoming~\eqref{eq:abstract_preconditioner}.
This non-symmetric solver (even when \(A\) is SPD) is the main motivation of
the proposed analysis and is investigated in more details in
Section~\ref{sec:ras}.
Alternatively, the symmetric (for a symmetric \(B_{j}\)) local solver
\begin{equation}\label{eq:as}
    S_{j}^{} := B_{j}^{-1},
    \qquad j=1,\dots,J,
\end{equation}
defines the Additive Schwarz (AS) method, discussed in Section~\ref{sec:as}.
An alternative symmetric (for a symmetric \(B_{j}\)) version of the RAS method
is
\begin{equation}\label{eq:soras}
    S_{j}^{} := D_{j}^{} B_{j}^{-1} D_{j}^{},
    \qquad j=1,\dots,J.
\end{equation}
For instance, the Symmetrized Optimized Restricted Additive Schwarz method
(SORAS) takes the form~\eqref{eq:soras} and is discussed in
Section~\ref{sec:soras}.

\paragraph{One-level preconditioner}

We can now define local preconditioners
\begin{equation}
    M_{j}^{-1} := R_{j}^{*} S_{j} R_{j}^{},
    \qquad j = 0, \dots, J.
\end{equation}
Note that despite the superscript ${}^{-1}$ these local contributions are not
invertible.
We consider a one-level preconditioner built additively using the local
preconditioners
\begin{align}
    \label{eq:abstract_preconditioner}
    \sum_{j=1}^{J} M_{j}^{-1}.
\end{align}
The purpose of the following analysis is to provide a somewhat systematic way
of constructing a global coarse space from only local contributions.
This should allow building scalable two-levels methods with a coarse space that
can be built in parallel.
To do this we focus first on the analysis of the one-level method in order to
identify the obstruction to its scalability.

\subsection{Extended decomposition}\label{sec:extension}

The one-level error propagation operator is
\begin{align}
    \label{eq:one-level-prop-error}
    I - \sum_{j=1}^{J} M_{j}^{-1} A
    & = I - \sum_{j=1}^{J} R_{j}^{*} S_{j} R_{j}^{} A
    = \sum_{j=1}^{J} R_{j}^{*} (D_{j}^{} R_{j}^{} - S_{j}^{} R_{j}^{} A).
\end{align}
Its analysis would be simplified if we could factor out the operator $R_{j}$ on
the right of each term in the sum.
Then, it would consist of purely local contributions.
More precisely, each term of the sum takes a global vector and outputs a global
vector whose support is purely local, namely in the range of \(R_{j}^{*}\) or
in other words with non-trivial elements only associated to \({\mathcal N}_j\).
To compute each local contribution, it is a priori necessary to know non-local
values of the input global vector (because of the matrix-vector product with
the matrix \(A\)), i.e.\  also values in the range of the projection \(I -
R_{j}^{*}R_{j}^{}\).
The forthcoming abstract analysis relies on the contrary that each local
contribution can be computed from local data.
It is possible to achieve this in a rather systematic way at the price of
extending the decomposition according to the adjacency graph of the matrix
\(A\), which is the motivation for the following definitions.

\paragraph{Extended decomposition}

We introduce another decomposition of the global set of indices \(\mathcal{N}\)
into \(J\) subsets \((\tilde{\mathcal{N}}_j)_{j=1,\dots,J}\) such that
\(\mathcal{N}_{j} \subset \tilde{\mathcal{N}}_{j}\).
This new decomposition into extended subdomains is characterized by their
associated Boolean restriction matrices
\begin{equation}
    \tilde{R}_{j}
    \in \mathbb{R}^{\# \tilde{\mathcal{N}}_{j} \times \# \mathcal{N}}, 
    \qquad j = 1, \dots, J.
\end{equation}
We also assume the following usual property for this decomposition
\begin{align}
    & \tilde{R}_{j}^{} \tilde{R}_{j}^{*} = \tilde{I}_j,
    \qquad j = 1, \dots, J,
\end{align}
where
\(\tilde{I}_j\in \mathbb{R}^{\# \tilde{\mathcal{N}}_{j} \times \# \tilde{\mathcal{N}}_{j}}\)
is the identity operator on the extended subdomain
\(\tilde{\mathcal{N}}_j\) that contains \(\mathcal{N}_j\).
Besides, we introduce the following operator that makes the link between the
two decompositions
\begin{equation}
    Q_{j} := R_{j}^{} \tilde{R}_{j}^{*}
    \in \mathbb{R}^{\# \mathcal{N}_j \times \# \tilde{\mathcal{N}}_j}, 
    \qquad j=1,\dots,J,
\end{equation}
by restricting data on \(\tilde{\mathcal{N}}_j\) to their values on  \(\mathcal{N}_j\).
To achieve the above-mentioned goal, we make the following assumption, which is not
difficult to satisfy in practice since it usually amounts to enlarging the
overlap (typically with one cell layer for a finite element discretization).
\begin{assumption}\label{hyp:tilde_decomposition}
    The two decompositions \(R_{j}^{}\) and \(\tilde{R}_{j}^{}\) are such that
    \begin{align}\label{eq:assumption_tilde_decomposition}
        R_{j}^{} = Q_{j}^{} \tilde{R}_{j}^{}\,
        \text{ and }\, 
        Q_{j}^{} \tilde{R}_{j}^{} A (I - \tilde{R}_{j}^{*} \tilde{R}_{j}^{})
        = R_{j}^{} A (I - \tilde{R}_{j}^{*} \tilde{R}_{j}^{}) = 0,
        \qquad j = 1, \dots, J.
    \end{align}
\end{assumption}
Note that the first identity in~\eqref{eq:assumption_tilde_decomposition} is
trivially satisfied if 
\(\mathcal{N}_{j} \subset \tilde{\mathcal{N}}_{j}\).

\paragraph{Extended partition of unity}

Using the partition of unity from the first decomposition we can define a
partition of unity associated to the second decomposition and characterized by
the diagonal matrices
\begin{equation}
    \tilde{D}_{j}
    := Q_{j}^{*} D_{j}^{} Q_{j}^{} \in \mathbb{R}^{\# \tilde{\mathcal{N}}_j \times \# \tilde{\mathcal{N}}_j},
    \qquad j=1,\dots,J.
\end{equation}
Using the first equation in~\eqref{eq:assumption_tilde_decomposition} it is
readily checked that they indeed form a partition of unity
\begin{align}
    I
    = \sum_{j=1}^{J} R_{j}^{*} D_{j} R_{j}^{}
    = \sum_{j=1}^{J} \tilde{R}_{j}^{*}
    Q_{j}^{*} D_{j}^{} Q_{j}^{}
    \tilde{R}_{j}^{}
    = \sum_{j=1}^{J} \tilde{R}_{j}^{*} \tilde{D}_{j}^{} \tilde{R}_{j}^{}.
\end{align}

\paragraph{Extended local solvers}

Similarly, we introduce local solvers associated to the second decomposition
\begin{equation}
    \tilde{S}_{j} := Q_{j}^{*} S_{j}^{} Q_{j}^{} \in \mathbb{C}^{\# \tilde{\mathcal{N}}_j \times \# \tilde{\mathcal{N}}_j},
    \qquad j=1,\dots,J.
\end{equation}
Note that, by definition, these matrices cannot be invertible even if \(S_{j}\) is.

\paragraph{New representations of preconditioners}

Using the two identities in~\eqref{eq:assumption_tilde_decomposition} we can
then give a different expression for the local preconditioners \(M_{j}^{-1}A\),
namely
\begin{align}
    M_{j}^{-1} A
    = R_{j}^{*} S_{j}^{} R_{j}^{} A
    = \tilde{R}_{j}^{*} Q_{j}^{*} S_{j}^{} Q_{j}^{} \tilde{R}_{j}^{} A \tilde{R}_{j}^{*} \tilde{R}_{j}^{} 
    = \tilde{R}_{j}^{*} \tilde{S}_{j}^{} \tilde{R}_{j}^{} A \tilde{R}_{j}^{*} \tilde{R}_{j}^{},
    \qquad j=1,\dots,J.
\end{align}
We therefore obtain a new expression for the one-level error propagation
operator
\begin{align}
    I - \sum_{j=1}^{J} M_{j}^{-1} A
    = \sum_{j=1}^{J}
    \tilde{R}_{j}^{*} \tilde{D}_{j}^{} \tilde{R}_{j}^{}
    - \tilde{R}_{j}^{*} \tilde{S}_{j}^{} \tilde{R}_{j}^{} A \tilde{R}_{j}^{*} \tilde{R}_{j}^{}
    = \sum_{j=1}^{J} \tilde{R}_{j}^{*} \tilde{L}_{j} \tilde{R}_{j}^{},
\end{align}
where we introduced
\begin{equation}
    \label{eq:Ltilde}
    \tilde{L}_{j} := 
    \tilde{D}_{j}^{} - \tilde{S}_{j}^{} \tilde{R}_{j}^{} A \tilde{R}_{j}^{*},
    \qquad j=1,\dots,J.
\end{equation}
The remarkable feature of the new expression of the first-level propagation
operator is that it is made of purely local contributions
\(\tilde{R}_{j}^{*} \tilde{L}_{j} \tilde{R}_{j}^{}\)
(the first operator applied is the restriction \(\tilde{R}_{j}^{}\)).
This property was achieved without making any assumptions on the local sovers,
only by adjusting (enlarging) the initial decomposition.

\subsection{Local generalized eigenvalue problems}

We consider here the one-level method.
The analysis is performed for a norm \(\|\cdot\|_{C}\) defined by a hermitian
positive definite matrix
\(C \in \mathbb{C}^{\# \mathcal{N} \times \# \mathcal{N}}\).
We make this choice to allow a priori problems with non SPD matrices \(A\),
which necessarily require measuring errors in a norm not defined by \(A\),
for instance some energy norm.
The reader can nevertheless safely assume \(C=A\) SPD in a first read.

We are interested in estimating the \(C\)-norm of the one-level error propagation
operator, namely
\begin{equation}
    \| I - \sum_{j=1}^{J} M_{j}^{-1} A \|_{C} 
    := \max_{\substack{u \in \mathbb{C}^{\#\mathcal{N}}\\u \neq 0}} \frac{
        \| (I - \sum_{j=1}^{J} M_{j}^{-1} A) u \|_{C}
    }{
        \|u\|_{C}
    }
    = \max_{\substack{u \in \mathbb{C}^{\#\mathcal{N}}\\u \neq 0}} \frac{
        \| \sum_{j=1}^{J} \tilde{R}_{j}^{*} \tilde{L}_{j}^{} \tilde{R}_{j}^{} u \|_{C}
    }{
        \|u\|_{C}
    },
\end{equation}
in terms of local quantities. We start with an estimation on the numerator based on
the following definition.
\begin{definition}\label{def:k0}
    Let \(k_{0}\) be the maximum multiplicity of the interactions (via the
    non-zero coefficients of the matrix \(C\)) between subdomains plus one.
\end{definition}
Following e.g.~\cite[Lem.~7.11]{Dolean:2015:IDDSiam},
for any \(u \in \mathbb{C}^{\#\mathcal{N}}\) we have:
\begin{equation}\label{eq:k0}
    \| \sum_{j=1}^{J} \tilde{R}_{j}^{*} \tilde{L}_{j}^{} \tilde{R}_{j}^{} u \|_{C}^{2}
    \leq k_{0}
    \sum_{j=1}^{J} \| \tilde{R}_{j}^{*} \tilde{L}_{j}^{} \tilde{R}_{j}^{} u \|_{C}^{2}.
\end{equation}
Note that a sharper estimate holds with $k_0$ replaced by the chromatic number
(the smallest number of colors needed to color the vertices so that no two
adjacent vertices share the same color) of the graph whose nodes are the
subdomains linked by an edge if and only if
\(\tilde{R}_{j}^{} C \tilde{R}_{i}^{*}  \neq 0\) for the subdomains $i$ and
$j$.
To estimate the denominator we make the following assumption.
\begin{assumption}\label{ass:k1}
    There exist local hermitian positive semi-definite operators
    \begin{equation}
        \tilde{C}_{j} \in \mathbb{C}^{\# \tilde{\mathcal{N}}_j \times \# \tilde{\mathcal{N}}_j},
        \qquad j=1,\dots,J,
    \end{equation}
    and some \(k_{1}>0\) such that, for all \(u\in \mathbb{C}^{\# \mathcal{N}}\),
    \begin{equation}\label{eq:spsd-splitting}
        \sum_{j=1}^{J} \left( \tilde{C}_{j} \tilde{R}_{j}^{} u, \tilde{R}_{j}^{} u \right)
        \leq k_{1} \|u\|_{C}^{2}.
    \end{equation}
\end{assumption}
The inequality~\eqref{eq:spsd-splitting} is sometimes referred to as a SPSD
splitting~\cite{alDaas:2019:class}.
If \(C\) stems from the finite element discretization of a hermitian coercive
operator, such operators \(\tilde{C}_{j}\) exist.
Indeed, taking the so-called Neumann matrices
\(\tilde{C}_{j}^{} := \tilde{C}_{j}^{N}\) of the subdomain \(j\),
then~\eqref{eq:spsd-splitting} holds for \(k_{1}\) the maximal multiplicity of
the subdomain intersection, i.e.\ the largest integer \(m\) such that there
exist \(m\) different subdomains whose intersection has a nonzero measure, see
e.g.\ \cite[Lem.~7.13]{Dolean:2015:IDDSiam}.
Similar constructions are also natural for finite volume schemes.

Using~\eqref{eq:k0} and~\eqref{eq:spsd-splitting}, we obtain the estimate
\begin{equation}
    \| I - \sum_{j=1}^{J} M_{j}^{-1} A \|_{C}^{2}
    \leq k_{0} k_{1} \max_{\substack{u \in \mathbb{C}^{\#\mathcal{N}}\\u \neq 0}} \frac{
        \sum_{j=1}^{J} \| \tilde{R}_{j}^{*} \tilde{L}_{j}^{} \tilde{R}_{j}^{} u \|_{C}^{2}
    }{
        \sum_{j=1}^{J} \left( \tilde{C}_{j} \tilde{R}_{j}^{} u, \tilde{R}_{j}^{} u \right)
    }.
\end{equation}
This estimate holds at least formally, since the right-hand-side might be
infinite if there exists \(u\) that makes the denominator vanish, which is a
priori possible in this general setting.

The next step of the derivation of the estimate makes use of the following
elementary lemma.
\begin{lemma}\label{lem:alphabeta}
    Let \(\alpha_{j}\) and \(\beta_{j}\) be non-negative real-valued
    functionals that satisfy the following assumption:
    \(
    \beta_{j}(u)=0
    \Rightarrow
    \alpha_{j}(u)=0, \forall u
    \). 
    Then, there holds
    \begin{equation}
        \frac{
            \sum_{j=1}^{J} \alpha_{j}(u)
        }{
            \sum_{j=1}^{J} \beta_{j}(u)
        }
        \le 
        \max_{j=1}^{J} \max_{v \text{ s.t. }\beta_{j}(v) \neq 0 } \frac{\alpha_{j}(v)}{\beta_{j}(v)},
        \qquad\forall u.
    \end{equation}
\end{lemma}
\begin{proof}
    The result stems from: for any \(u\)
    \begin{equation}
        \sum_{j=1}^{J} \alpha_{j}(u)
        \le 
        \sum_{j=1}^{J} \beta_{j}(u) \max_{v \text{ s.t. }\beta_{j}(v) \neq 0 }\frac{\alpha_{j}(v)}{\beta_{j}(v)}
        \le 
        \left(\sum_{j=1}^{J} \beta_{j}(u) \right) \max_{k=1}^{J} \max_{v \text{ s.t. } \beta_{k}(v) \neq 0 }\frac{\alpha_{k}(v)}{\beta_{k}(v)}.
    \end{equation}
\end{proof}

In order to ease the analysis we assume only here that the operators
\(\tilde{C}_{j}^{}\) are invertible.
Then, using Lemma~\ref{lem:alphabeta}, we get:
\begin{equation}\label{eq:upper_bound_onelvl_estimate}
    \| I - \sum_{j=1}^{J} M_{j}^{-1} A \|_{C}^{2}
    \leq k_{0} k_{1}
    \max_{\substack{u \in \mathbb{C}^{\#\mathcal{N}}\\u \neq 0}}
    \max_{j=1}^{J}
    \frac{
        \| \tilde{R}_{j}^{*} \tilde{L}_{j}^{} \tilde{R}_{j}^{} u \|_{C}^{2}
    }{
        \left( \tilde{C}_{j}^{} \tilde{R}_{j}^{} u, \tilde{R}_{j}^{} u \right)
    }
    \le k_{0} k_{1}
    \max_{j=1}^{J}
    \max_{\substack{u_{j} \in \mathbb{C}^{\#\tilde{\mathcal{N}}_{j}}\\u_{j} \neq 0}}
    \frac{
        \| \tilde{R}_{j}^{*} \tilde{L}_{j}^{} u_{j} \|_{C}^{2}
    }{
        \left( \tilde{C}_{j} u_{j}, u_{j} \right)
    }.
\end{equation}
The above estimate can be somehow controlled by only local considerations.
This was the motivation to introduce the extended decomposition in
Section~\ref{sec:extension}.
The estimate prompts us to consider the following local generalized eigenvalue
problems
\begin{equation}\label{eq:gevp}
   \begin{cases}
        \text{Find}\ 
        (\lambda_{j},u_{j}) \in \mathbb{R} \times
        \mathbb{C}^{\#\tilde{\mathcal{N}}_{j}}
        \ \text{such that}:\\
        \tilde{L}_{j}^{*} \tilde{R}_{j}^{} C \tilde{R}_{j}^{*} \tilde{L}_{j}^{} u_{j}
        = \lambda_{j} \tilde{C}_{j} u_{j},
   \end{cases} 
   \qquad j=1,\dots,J,
\end{equation}
where \(\tilde{L}_{j}\) is defined in~\eqref{eq:Ltilde}.
In the following we are interested in characterizing eigenvectors associated to
large eigenvalues, which are responsible for a large upper bound
in~\eqref{eq:upper_bound_onelvl_estimate}.

\begin{remark}\label{rem:eigenproblem_kernels}
    The local generalized eigenvalue problems~\eqref{eq:gevp} make sense for
    an eigenvector
    \(
        u_{j} \notin 
        \ker \tilde{C}_{j}
        \cap
        \ker \tilde{L}_{j}^{*} \tilde{R}_{j}^{} C \tilde{R}_{j}^{*} \tilde{L}_{j}^{}
    \)
    (which is usually a trivial set in practice)
    otherwise any eigenvalue \(\lambda_{j} \in \mathbb{R}\) satisfies the equation.
    Besides, a non-trivial eigenvector \( u_{j} \in \ker \tilde{C}_{j} \)
    (which usually has a small dimension in practice)
    but
    \(
        u_{j} \notin 
        \ker \tilde{L}_{j}^{*} \tilde{R}_{j}^{} C \tilde{R}_{j}^{*} \tilde{L}_{j}^{}
    \)
    implies formally \(\lambda_{j} = +\infty\).
    Finally, a non-trivial eigenvector
    \(
        u_{j} \in 
        \ker \tilde{L}_{j}^{*} \tilde{R}_{j}^{} C \tilde{R}_{j}^{*} \tilde{L}_{j}^{}
    \)
    but
    \( u_{j} \notin \ker \tilde{C}_{j} \)
    implies formally \(\lambda_{j} = 0\).
\end{remark}

Let us denote by \(P_{\tilde{C}_{j}}\) the orthogonal (for the Euclidean inner
product) projection on \(\operatorname*{range} \tilde{C}_{j}\).
The following definition is motivated by Remark~\ref{rem:eigenproblem_kernels}.
\begin{definition}\label{def:extendedGenEO}
    The local extended generalized eigenvalue problems we consider are defined as, for \(j=1,\dots,J\):
    \begin{equation}\label{eq:gevp_generic}
       \begin{cases}
            \text{Find}\ 
            (\lambda_{j},u_{j}) \in \mathbb{R} \times \operatorname*{range} \tilde{C}_{j}
            \ \text{such that}:\\
            P_{\tilde{C}_{j}}
            ( \tilde{D}_{j}^{} - \tilde{S}_{j}^{} \tilde{R}_{j}^{} A \tilde{R}_{j}^{*} )^{*}
            ( \tilde{R}_{j}^{} C \tilde{R}_{j}^{*} )
            ( \tilde{D}_{j}^{} - \tilde{S}_{j}^{} \tilde{R}_{j}^{} A \tilde{R}_{j}^{*} )^{}
            P_{\tilde{C}_{j}}
            u_{j}
            = \lambda_{j} \tilde{C}_{j} u_{j}.
       \end{cases} 
    \end{equation}
\end{definition}
Since \(C\) is assumed to be Hermitian positive definite, for each subdomain $j$ the matrix
featured in the left-hand-side of~\eqref{eq:gevp_generic}, namely
\(
    ( \tilde{D}_{j}^{} - \tilde{S}_{j}^{} \tilde{R}_{j}^{} A \tilde{R}_{j}^{*} )^{*}
    ( \tilde{R}_{j}^{} C \tilde{R}_{j}^{*} )
    ( \tilde{D}_{j}^{} - \tilde{S}_{j}^{} \tilde{R}_{j}^{} A \tilde{R}_{j}^{*} )^{}
\),
is Hermitian positive semi-definite.
While the problems~\eqref{eq:gevp_generic} are posed on the extended
decomposition, it is possible to equivalently express the local generalized
eigenvalue problems in a form that can be rewritten using the initial
decomposition operators, for \(j=1,\dots,J\):
\begin{equation}\label{eq:gevp_generic_smaller_decomposition}
   \begin{cases}
        \text{Find}\ 
        (\lambda_{j},u_{j}) \in \mathbb{R} \times \operatorname*{range} \tilde{C}_{j}
        \ \text{such that}:\\
        P_{\tilde{C}_{j}}
        \tilde{R}_{j}^{} 
        ( {D}_{j}^{} {R}_{j}^{} - {S}_{j}^{} {R}_{j}^{} A )^{*}
        ( {R}_{j}^{} C {R}_{j}^{*} )
        ( {D}_{j}^{} {R}_{j}^{} - {S}_{j}^{} {R}_{j}^{} A )^{}
        \tilde{R}_{j}^{*}
        P_{\tilde{C}_{j}}
        u_{j}
        = \lambda_{j} \tilde{C}_{j} u_{j}.
   \end{cases} 
\end{equation}
If the local solvers take the form \({S}_{j} = {D}_j {B}_{j}^{-1}\)
for some invertible \({B}_{j}\) (e.g.\ \({B}_{j} := R_{j}^{} A  R_{j}^{*}\),
which is a common choice used for instance with RAS preconditioners, see
Section~\ref{sec:ras} for more details), then the local generalized eigenvalue
problems are, for \(j=1,\dots,J\):
\begin{equation}
   \begin{cases}
        \text{Find}\ 
        (\lambda_{j},u_{j}) \in \mathbb{R} \times \operatorname*{range} \tilde{C}_{j}
        \ \text{such that}:\\
        P_{\tilde{C}_{j}}
        \tilde{R}_{j}^{} 
        ( {R}_{j}^{} - {B}_{j}^{-1} {R}_{j}^{} A )^{*}
        ( {D}_{j}^{} {R}_{j}^{} C {R}_{j}^{*} {D}_{j}^{} )
        ( {R}_{j}^{} - {B}_{j}^{-1} {R}_{j}^{} A )^{}
        \tilde{R}_{j}^{*} 
        P_{\tilde{C}_{j}}
        u_{j}
        = \lambda_{j} \tilde{C}_{j} u_{j}.
   \end{cases} 
\end{equation}

\subsection{Extended GenEO coarse space}\label{sec:new_cs}

Based on the analysis of the one-level method described above, we are able to
design a coarse space such that the norm of the propagation error of the
corresponding two-level method is smaller than one.

The following definition is motivated by Remark~\ref{rem:eigenproblem_kernels}.
\begin{definition}\label{def:Yj}
    Let \(Y_{j}\) denote the orthogonal (for the Euclidean inner product) complementary in
    \( \ker \tilde{C}_{j} \)
    of
    \(
        \ker \tilde{C}_{j}
        \cap
        \ker \tilde{L}_{j}^{*} \tilde{R}_{j}^{} C \tilde{R}_{j}^{*} \tilde{L}_{j}^{}
    \).
\end{definition}
This space is such that the following orthogonal decomposition holds
\begin{equation}
    \ker \tilde{C}_{j}
    =
    \left(
        \ker \tilde{C}_{j}
        \cap
        \ker \tilde{L}_{j}^{*} \tilde{R}_{j}^{} C \tilde{R}_{j}^{*} \tilde{L}_{j}^{}
    \right)
    \oplus_{\perp}
    Y_{j},
    \qquad j=1,\dots,J.
\end{equation}
\begin{definition}
    \label{def:newCS}
    Let \(\tau > 0\) be a user-specified parameter.
    Based on the generalized eigenvalue problems~\eqref{eq:gevp_generic}, we
    define for each subdomain $j=1,\dots,J$ the local spaces $U_{j}$:
    \begin{equation}
        U_{j} :=
        Y_{j} \bigoplus
        \operatorname*{span}
        \left\{
            u_{j} \;|\;
            (\lambda_{j}, u_{j}) \ \text{solution of~\eqref{eq:gevp_generic} with }
            \lambda_{j} > \tau
        \right\}.
    \end{equation}
    Then, the global coarse space is defined as the collection of all the
    local contributions from a subdomain \(j\) to which the operator $
    \tilde{R}_{j}^{*} \tilde{L}_{j}^{}$ is applied:
    \begin{equation}
        Z :=
        \bigoplus_{j=1}^{J}
        \tilde{R}_{j}^{*} \tilde{L}_{j}^{} U_{j} =
        \bigoplus_{j=1}^{J}
        \tilde{R}_{j}^{*}
        ( \tilde{D}_{j}^{} - \tilde{S}_{j}^{} \tilde{R}_{j}^{} A \tilde{R}_{j}^{*} )^{}
        U_{j}.
    \end{equation}
\end{definition}
We introduce the projections \(\Pi_{j}\) which are key components in the
analysis of the coarse space but which are not used in practice to define the
preconditioner itself.
\begin{definition}\label{def:Pii}
    Let \(\Pi_{j}\) be the projection on \(U_{j}\) parallel to 
    \begin{equation}
        \left(
            \ker \tilde{C}_{j}
            \cap
            \ker \tilde{L}_{j}^{*} \tilde{R}_{j}^{} C \tilde{R}_{j}^{*} \tilde{L}_{j}^{}
        \right)
        \bigoplus
        \operatorname*{span}
        \left\{
            u_{j} \;|\;
            (\lambda_{j}, u_{j}) \ \text{solution of~\eqref{eq:gevp_generic} with }
            \lambda_{j} \leq \tau
        \right\}.
    \end{equation}
\end{definition}
Following~\cite[Lem.~2.3]{alDaas:2019:class}, we have,
for any \(u_{j} \in \mathbb{C}^{\# \tilde{\mathcal{N}}_{j}}\)
\begin{equation}\label{eq:ev_tau_estimate}
    \| \tilde{R}_{j}^{*} \tilde{L}_{j}^{} (I - \Pi_{j}) u_{j} \|_{C}^{2}
    \leq \tau
    \left( \tilde{C}_{j} u_{j}, u_{j} \right),
    \qquad j=1,\dots,J.
\end{equation}
We introduce the matrix \(R_{0}\), such that the columns of \(R_{0}^{*}\) form a basis of \(Z\)
and make the following assumption.

\begin{assumption}\label{hyp:R0AR0_invertible}
    We assume in the following that \(R_{0}^{} A R_{0}^{*}\) is invertible.
\end{assumption}

The Assumption~\ref{hyp:R0AR0_invertible} is satisfied for instance if the
symmetric part of \(A\), namely \(\frac{A+A^{*}}{2}\) is SPD.
Thanks to Assumption~\ref{hyp:R0AR0_invertible} we let:
\begin{equation}
    \label{eq:M0m1}
    M_{0}^{-1} := R_{0}^{*} (R_{0}^{} A R_{0}^{*})^{-1} R_{0}^{}.
\end{equation}
The operator \(M_{0}^{-1} A\) is then a projection with:
\begin{equation}
    \operatorname*{range} (M_{0}^{-1} A) = Z.
\end{equation}

\subsection{Two-level method analysis}

We consider a two-level preconditioner \(M^{-1}\) constructed by adding a
multiplicative global coarse correction to the first-level residual propagation
operator.
More precisely, the full preconditioner is defined by
\begin{equation}\label{eq:full_preconditioner}
    M^{-1}
    := \sum_{j=1}^{J} M_{j}^{-1}
    + M_{0}^{-1} (I - A \sum_{j=1}^{J} M_{j}^{-1})
    = M_{0}^{-1} + (I - M_{0}^{-1} A) \sum_{j=1}^{J} M_{j}^{-1}.
\end{equation}
Hence the two-level error propagation operator is given by
\begin{equation}\label{eq:full_error_op}
    I - M^{-1} A
    = I - \sum_{j=1}^{J} M_{j}^{-1} A - M_{0}^{-1} A (I - \sum_{j=1}^{J} M_{j}^{-1} A) 
    = (I - M_{0}^{-1} A) (I - \sum_{j=1}^{J} M_{j}^{-1} A).
\end{equation}

Before stating our made result we introduce and discuss the behavior of the
following quantity: let 
\begin{equation}\label{eq:M0goodpreconditioner}
    \sigma := \|I - M_0^{-1} A\|_{C}.
\end{equation}
We claim that \(\sigma\) can be controlled in the following cases:
\begin{itemize}
    \item If \(A\) is SPD, then choosing \(C=A\) immediately yields \(\sigma = 1\),
        since \(M_{0}^{-1}A\) and \(I - M_{0}^{-1}A\) are then
        complementary \(A\)-orthogonal projections.
    \item If the coarse space has the same rank as \(A\), then \(\sigma=0\). 
        Even though this situation is of no practical interest, one might
        expect that, for large enough coarse spaces, \(\sigma\) is moderate.
        Theoretical guarantees ensuring control of \(\sigma\) for Helmholtz
        problems are given for instance in~\cite[Lemma~5.1]{Galkowski2025}.
    \item If \(C\) is spectrally equivalent to \(A\), then
        \(\sigma\) is controlled, which is the subject of the following lemma.
\end{itemize}

\begin{lemma}\label{lem:norm_second_lvl}
    Let \(\rho \geq 0\) such that
    \begin{equation}\label{eq:C_good_A_preconditioner}
        \rho := \|I - C^{-1}A\|_{C}.
    \end{equation}
    Assume that \(\rho<1\), then 
    \begin{equation}\label{eq:norm_second_lvl}
         \|I - M_{0}^{-1} A\|_C \leq \frac{1}{1-\rho}.
    \end{equation}
    
\end{lemma}
\begin{proof}
    Letting $P_C$ denote the $C$-orthogonal projection onto the coarse space \(Z\), 
    namely \(P_C:= R_{0}^{*} (R_{0}^{} C R_{0}^{*})^{-1} R_{0}^{} C\),
    we write first the following decomposition of the coarse space operator
    \begin{align}\label{eq:lem_norm_second_lvl_first_eq}
        I - M_{0}^{-1} A
        &= (I - P_{C})
        + P_{C}(I - C^{-1}A)
        + (P_{C}C^{-1}A - M_{0}^{-1} A).
    \end{align}
    We focus on the third term in the right-hand-side, we have
    \begin{align}
        P_{C}C^{-1}A - M_{0}^{-1} A
        &= R_{0}^{*} (R_{0}^{} C R_{0}^{*})^{-1} R_{0}^{} A - R_{0}^{*} (R_{0}^{} A R_{0}^{*})^{-1} R_{0}^{} A \\
        &= R_{0}^{*} \left[ (R_{0}^{} C R_{0}^{*})^{-1}  (R_{0}^{} A R_{0}^{*} - R_{0}^{} C R_{0}^{*})  (R_{0}^{} A R_{0}^{*})^{-1} \right] R_{0}^{} A \\
        &= R_{0}^{*} (R_{0}^{} C R_{0}^{*})^{-1}  R_{0}^{}C (C^{-1} A - I) R_{0}^{*} (R_{0}^{} A R_{0}^{*})^{-1} R_{0}^{} A\\
        &= -P_{C} (I - C^{-1} A) M_0^{-1} A.
    \end{align}
    Plugging in~\eqref{eq:lem_norm_second_lvl_first_eq}, we obtain
    \begin{align}
        I - M_{0}^{-1} A
        = (I - P_{C}) + P_{C}(I - C^{-1}A) (I - M_0^{-1} A).
    \end{align}
    The claimed result then follows, since,
    using that \(\|P_{C}\|_{C} = \|I - P_{C}\|_{C} = 1\), we have
    \begin{align}
        \|I - M_{0}^{-1} A\|_{C}
        &\leq \|I - P_{C}\|_{C} + \|P_{C}\|_{C} \|I - C^{-1}A\|_{C} \|I - M_0^{-1} A\|_{C}\\
        &\leq 1 + \rho \|I - M_0^{-1} A\|_{C}.
    \end{align}
\end{proof}

The quantity \(\rho\) defined in~\eqref{eq:C_good_A_preconditioner}
measures how good \(C^{-1}\) is a preconditioner for \(A\).
This is a question strictly related to the problem under consideration and in
particular \(\rho\) does not depend on the domain decomposition.
Note that if \(A\) is SPD and \(C=A\) then \(\rho=0\).

We are now ready to state the main result of the paper.

\begin{theorem}
    Under Assumption~\ref{hyp:R0AR0_invertible}, let $M_{0}^{-1}$ be defined
    by~\eqref{eq:M0m1} for the coarse space of Definition~\ref{def:newCS} for
    some user-chosen parameter $\tau$.
    Let \(M^{-1}\) defined in~\eqref{eq:full_preconditioner} denote the
    corresponding two-level preconditioner.

    Then, we have the following estimate:
        \begin{equation}\label{eq:op_norm_estimate}
            \| I - M_{}^{-1} A \|_{C} \leq \sigma\sqrt{k_{0} k_{1} \tau},
        \end{equation}
    where $\sigma$ is defined by~\eqref{eq:M0goodpreconditioner}, $k_0$ by
    Definition~\ref{def:k0} and $k_1$ by Assumption~\ref{ass:k1}. 
\end{theorem}
\begin{proof}
    We need to estimate the \(C\)-norm of the two-level error propagation
    operator, namely
    \begin{equation}
        \| I - M_{}^{-1} A \|_{C} 
        := \max_{\substack{u \in \mathbb{C}^{\#\mathcal{N}}\\u \neq 0}} \frac{
            \| (I - M_{}^{-1} A) u \|_{C}
        }{
            \|u\|_{C}
        }.
    \end{equation}
    To study this quantity we rewrite first the one-level error propagation operator
    using the local projection operators
    \begin{align}
        I - \sum_{j=1}^{J} M_{j}^{-1} A
        = \sum_{j=1}^{J} \tilde{R}_{j}^{*} \tilde{L}_{j}^{} \tilde{R}_{j}^{}
        = \sum_{j=1}^{J} \tilde{R}_{j}^{*} \tilde{L}_{j}^{} (I-\Pi_{j}) \tilde{R}_{j}^{}
        + \sum_{j=1}^{J} \tilde{R}_{j}^{*} \tilde{L}_{j}^{} \Pi_{j} \tilde{R}_{j}^{}.
    \end{align}
    By definition of the coarse space \(Z\) and the projections \(\Pi_{j}\),
    \begin{equation}
        \operatorname*{range}(\sum_{j=1}^{J} \tilde{R}_{j}^{*} \tilde{L}_{j}^{} \Pi_{j} \tilde{R}_{j}^{})
        \subset Z = \ker (I - M_{0}^{-1} A).
    \end{equation}
    It follows that the two-level error propagation
    operator~\eqref{eq:full_error_op} can be rewritten as
    \begin{align}
        I - M_{}^{-1} A
        & = (I - M_{0}^{-1} A) (I - \sum_{j=1}^{J} M_{j}^{-1} A)
        = (I - M_{0}^{-1} A)
        (\sum_{j=1}^{J} \tilde{R}_{j}^{*} \tilde{L}_{j}^{} (I-\Pi_{j}) \tilde{R}_{j}^{}).
    \end{align}
    By definition~\eqref{eq:M0goodpreconditioner} of $\sigma$, we have:
    \begin{align}
        \| I - M_{}^{-1} A \|_{C}
        &\leq \| (I - M_{0}^{-1} A)\|_C \| \sum_{j=1}^{J} \tilde{R}_{j}^{*} \tilde{L}_{j}^{} (I-\Pi_{j}) \tilde{R}_{j}^{} \|_{C}\\
        &\leq \sigma \;
        \max_{\substack{u \in \mathbb{C}^{\#\mathcal{N}}\\u \neq 0}} \frac{
            \| \sum_{j=1}^{J} \tilde{R}_{j}^{*} \tilde{L}_{j}^{} (I-\Pi_{j}) \tilde{R}_{j}^{} u \|_{C}       
        }{
            \| u \|_{C}
        }.
    \end{align}
    The claimed result is now essentially obtained by repeating the analysis
    that was performed for the one-level method:
    for any \(u \in \mathbb{C}^{\#\mathcal{N}}\), we have
    \begin{align}
        \| \sum_{j=1}^{J} \tilde{R}_{j}^{*} \tilde{L}_{j}^{} (I-\Pi_{j}) \tilde{R}_{j}^{} u \|_{C}^{2}
        & \underset{\eqref{eq:k0}}{\leq} k_{0}
        \sum_{j=1}^{J} \| \tilde{R}_{j}^{*} \tilde{L}_{j}^{} (I-\Pi_{j}) \tilde{R}_{j}^{} u \|_{C}^{2}\\
        & \underset{\eqref{eq:ev_tau_estimate}}{\leq} k_{0} \tau
        \sum_{j=1}^{J} \left( \tilde{C}_{j} \tilde{R}_{j}^{} u, \tilde{R}_{j}^{} u \right)
        \underset{\eqref{eq:spsd-splitting}}{\leq} k_{0} k_{1} \tau
        \|  u \|_{C}^{2}.
    \end{align}
 \end{proof}

\begin{corollary}
    Assuming that \(\tau\) is chosen sufficiently small such that
    \begin{equation}
        \sigma \sqrt{k_{0} k_{1} \tau} < 1,
    \end{equation}
    then, for any initial guess \(u^{0} \in \mathbb{C}^{\# \mathcal{N}}\), the
    fixed point algorithm
    \begin{equation}
      u^{n+1} := u^{n} + M^{-1}(f-A\,u^{n}),
      \qquad n \in \mathbb{N},
    \end{equation}
    converges in the \(C\)-norm to the solution of the linear
    system~\eqref{eq:syslin} with a convergence rate of at most
    \(\sigma\sqrt{k_{0} k_{1} \tau}\).

    Then, also, the preconditioned operator \(M^{-1}A\) (where \(M^{-1}\) is
    the two-level preconditioner defined by~\eqref{eq:full_preconditioner}) is
    coercive
    \begin{equation}\label{eq:op_coercivity_estimate}
        \left( M_{}^{-1} Au, u \right)_{C} \geq
        \left[1-\sigma\sqrt{k_{0} k_{1} \tau}\right] \| u \|_{C}^{2},
        \qquad \forall u \in \mathbb{C}^{\#\mathcal{N}}.
    \end{equation}
\end{corollary}
\begin{proof}
    The convergence of the fixed point algorithm is clear.
    As for the coercivity, by the triangular and Cauchy-Schwarz inequalities,
    we have, for any \(u \in \mathbb{C}^{\#\mathcal{N}}\):
   \begin{align}
       \left( M_{}^{-1} Au, u \right)_{C}
       &= \|u\|_{C}^{2} - \left(u, (I-M_{}^{-1} A)u\right)_{C}
       \ge \|u\|_{C}^{2} - \|u\|_{C} \|(I-M_{}^{-1} A)u\|_{C}\\
       &\ge [1-\sigma\sqrt{k_{0} k_{1} \tau}]\, \|u\|_{C}^{2}.  
   \end{align}
\end{proof}

\section{Practical methods}\label{sec:practical_methods}

We investigate how some particular choices of local solvers in
the first-level preconditioner influence the structure of the local generalized
eigenvalue problems and as a result the associated coarse space. 
Considering first a RAS preconditioner, we show in~\S~\ref{sec:ras_exact} that
with exact local solvers ($B_{j}:=R_{j}^{} A R_{j}^{*}$), the contribution of
the subdomains to the coarse space is made of local \(A\)-harmonic vectors
multiplied by the partition of unity and extended by zero,
see~\eqref{eq:csrasexactsolves}.
In~\S~\ref{sec:ras_harmonic}, we show that if the matrices $C$ and
$\tilde{C}_{j}$ are defined from the matrix $A$, then the Extended GenEO
eigenvectors of Definition~\ref{def:extendedGenEO} are themselves discrete
\(A\)-harmonic vectors by construction, hence the eigenproblems can be solved
in a smaller space.
The form of the local generalized eigenvalue problems and the coarse spaces for
the AS method and a SORAS preconditioner are then detailed in \S~\ref{sec:as}
and \S~\ref{sec:soras} respectively.

\subsection{Restricted Additive Schwarz}\label{sec:ras}

\subsubsection{RAS with inexact solvers}\label{sec:ras_inexact}

A key feature of our abstract analysis is that it covers non-symmetric
preconditioners, which is the case of the Restricted Additive Schwarz (RAS)
method corresponding to the local solvers~\eqref{eq:ras}, namely
\( S_{j}^{} := D_{j}^{} B_{j}^{-1}\).
The extended local solvers \(\tilde{S}_{j}\) are then, for \(j=1,\dots,J\):
\begin{equation}
    \tilde{S}_{j}
    := Q_{j}^{*} S_{j}^{} Q_{j}
    = Q_{j}^{*} D_{j}^{} B_{j}^{-1} Q_{j}
    = Q_{j}^{*} D_{j}^{} Q_{j}^{} Q_{j}^{*} B_{j}^{-1} Q_{j}
    = \tilde{D}_{j}^{} Q_{j}^{*} B_{j}^{-1} Q_{j}.
\end{equation}
We therefore obtain, for \(j=1,\dots,J\):
\begin{equation}
    \tilde{L}_{j} := 
    \tilde{D}_{j}^{} - \tilde{S}_{j}^{} \tilde{R}_{j}^{} A \tilde{R}_{j}^{*} 
    = \tilde{D}_{j}^{} (\tilde{I}_{j} - Q_{j}^{*} B_{j}^{-1} Q_{j}^{} \tilde{R}_{j}^{} A \tilde{R}_{j}^{*}).
\end{equation}
where \(\tilde{I}_j \in \mathbb{R}^{\# {\tilde{\mathcal{N}}_j}\times
\#\tilde{\mathcal{N}}_{j}}\) is the identity operator on the subdomain defined
by \(\tilde{\mathcal{N}}_{j}\).
The local generalized eigenvalue problems~\eqref{eq:gevp_generic} take in this
case the particular form of, for \(j=1,\dots,J\),
\begin{equation}\label{eq:gevp_ras}
   \begin{cases}
      \text{Find}\ 
      (\lambda_{j},u_{j}) \in \mathbb{R} \times
      \operatorname*{range} \tilde{C}_{j}
      \ \text{such that:}\\
      P_{\tilde{C}_{j}}
      ( \tilde{I}_{j} - Q_{j}^{*} B_{j}^{-1} Q_{j}^{} \tilde{R}_{j}^{} A \tilde{R}_{j}^{*} )^{*}
      \tilde{D}_{j}^{} ( \tilde{R}_{j}^{} C \tilde{R}_{j}^{*} ) \tilde{D}_{j}^{}
      ( \tilde{I}_{j} - Q_{j}^{*} B_{j}^{-1} Q_{j}^{} \tilde{R}_{j}^{} A \tilde{R}_{j}^{*} )^{}
      P_{\tilde{C}_{j}}
      u_{j}
      = \lambda_{j} \tilde{C}_{j} u_{j}.
   \end{cases} 
\end{equation}

\subsubsection{RAS with exact solvers yields locally harmonic eigenvectors}\label{sec:ras_exact}

We consider exact local solvers for the small decomposition \(R_{j}^{}\),
namely we set
\begin{equation}
    B_{j}^{}
    := R_{j}^{} A R_{j}^{*}
     = Q_{j}^{} \tilde{R}_{j}^{} A \tilde{R}_{j}^{*} Q_{j}^{*},
   \qquad j=1,\dots,J,
\end{equation}
and assume that $B_{j}^{}$ is invertible.

\begin{lemma}\label{lem:exactlocalsolvesHarmonic}
    We introduce the projection \(\tilde{P}_{j}^{}\):
    \begin{equation}
        \tilde{P}_{j}^{} := Q_{j}^{*}
        (Q_{j}^{} \tilde{R}_{j}^{} A \tilde{R}_{j}^{*} Q_{j}^{*})^{-1}
        Q_{j}^{} \tilde{R}_{j}^{} A \tilde{R}_{j}^{*},
        \qquad j=1,\dots,J.
    \end{equation}
    Then, 
    \(
        \operatorname*{range}(\tilde{I}_{j} -
        \tilde{P}_{j})=\operatorname*{ker}(\tilde{P}_{j})
    \)
    is a space of discrete $A$-harmonic vectors
\end{lemma}
\begin{proof}
    Consider \(\tilde{u}_{j}\) in the kernel of \(\tilde{P}_{j}\).
    If we left-multiply the equation \(\tilde{P}_{j}\tilde{u}_{j}=0\) by the
    operator \(Q_{j}^{} \tilde{R}_{j}^{} A \tilde{R}_{j}^{*}\), we obtain that
    \(Q_{j}^{} \tilde{R}_{j}^{} A \tilde{R}_{j}^{*}\,\tilde{u}_{j}=0\), that is
    to say that \(\tilde{R}_{j}^{} A \tilde{R}_{j}^{*}\,\tilde{u}_{j}\) is zero
    in the subdomain \(\mathcal{N}_{j}\).
    Thus, we can say that a vector in the kernel of the projection
    \(\tilde{P}_{j}\) is a discrete \(A\)-harmonic function. 
\end{proof}

Since we have
\(
    \tilde{L}_{j} = \tilde{D}_{j}^{} (\tilde{I}_{j} - \tilde{P}_{j}^{})
\),
\(
    j=1,\dots,J,
\)
the local generalized eigenvalue problems~\eqref{eq:gevp} reads:
\begin{equation}\label{eq:gevp_ras_proj}
   \begin{cases}
        \text{Find}\ 
        (\lambda_{j},u_{j}) \in \mathbb{R} \times
        \operatorname*{range} \tilde{C}_{j}^{}
        \ \text{such that}:\\
        P_{\tilde{C}_{j}}
        ( \tilde{I}_{j} - \tilde{P}_{j}^{} )^{*}
        \tilde{D}_{j}^{} ( \tilde{R}_{j}^{} C \tilde{R}_{j}^{*} ) \tilde{D}_{j}^{}
        ( \tilde{I}_{j} - \tilde{P}_{j}^{} )^{}
        P_{\tilde{C}_{j}}
        u_{j}
        = \lambda_{j} \tilde{C}_{j}^{} u_{j},
   \end{cases} 
   \quad j=1,\dots,J.
\end{equation}
The coarse space introduced in Definition~\ref{def:newCS} reads:
\begin{equation}
    \label{eq:csrasexactsolves}
    Z := \bigoplus_{j=1}^{J}
    \tilde{R}_{j}^{*} \tilde{D}_{j}^{} ( \tilde{I}_{j} - \tilde{P}_{j}^{} )
    \left[
    Y_{j}^{} \bigoplus
    \operatorname*{span}
    \left\{
        u_{j} \;|\;
        (\lambda_{j}, u_{j}) \ \text{solution of~\eqref{eq:gevp_ras_proj} with }
        \lambda_{j} > \tau
    \right\}
    \right],
\end{equation}
where \(Y_{j}^{}\) is introduced in Definition~\ref{def:Yj}.
In words, from Lemma~\ref{lem:exactlocalsolvesHarmonic}, the coarse space is
populated by \(A\)-harmonic vectors weighted by a partition of unity and then
extended by zero.

\subsubsection{RAS with exact solvers in the SPD case}\label{sec:ras_harmonic}

\begin{assumption}
    \label{ass:CequalsA}
    We choose \(C:=A\) and assume that \(A\) is SPD.
    For simplicity, we consider a SPSD splitting~\eqref{eq:spsd-splitting}
    with invertible \(\tilde{C}_{j}^{} \) so that
    \(
        \operatorname*{range} \tilde{C}_{j}^{} =
        \mathbb{C}^{\#\tilde{\mathcal{N}}_{j}}
    \)
    and thus \(P_{\tilde{C}_{j}}=\tilde{I}_{j}\).
    Finally, we assume that the identity
    $Q_j \tilde{C}_{j}^{}= Q_j \tilde{R}_{j}^{} A \tilde{R}_{j}^{*}$ holds.
\end{assumption}

In plain words, the assumption implies that the rows of \(\tilde{C}_{j}^{}\)
and of \(\tilde{R}_{j}A\tilde{R}_{j}^{*}\) are identical for the interior
points of the subdomain \(j\).
This is obviously the case if \(\tilde{C}_{j}^{}=\tilde{R}_{j}^{} A
\tilde{R}_{j}^{*}\) (see~\cite{AlDaas:2024:robust} for a choice very close to
this one) or if \(\tilde{C}_{j}^{}\) is the so-called local Neumann matrix as
in the standard GenEO preconditioner,
see~\cite{Spillane:2014:ASC,Dolean:2015:IDDSiam}.
In the former case, the construction is purely algebraic but the constant
\(k_1\) (see Definition~\ref{ass:k1}) can be quite large, whereas in the latter
case, it is necessary to know the elementary matrices (used to assemble the
matrix $A$) but then \(k_1\) is only the maximum multiplicity of the subdomain
intersections. 

Under these assumptions, 
\begin{equation}
    \tilde{P}_{j}^{} := Q_{j}^{*}
    (Q_{j}^{} \tilde{R}_{j}^{} A \tilde{R}_{j}^{*} Q_{j}^{*})^{-1}
    Q_{j}^{} \tilde{R}_{j}^{} A \tilde{R}_{j}^{*}
    = Q_{j}^{*}
    (Q_{j}^{} \tilde{C}_{j}^{} Q_{j}^{*})^{-1}
    Q_{j}^{} \tilde{C}_{j}^{}
    ,
    \qquad j=1,\dots,J,
\end{equation}
and thus the local generalized eigenvalue
problems~\eqref{eq:gevp_ras_proj} are 
\begin{equation}\label{eq:gevp_ras_proj_CANeu}
   \begin{cases}
        \text{Find}\ 
        (\lambda_{j},u_{j}) \in \mathbb{R} \times
        \mathbb{C}^{\#\tilde{\mathcal{N}}_{j}}
        \ \text{such that}:\\
        ( \tilde{I}_{j} - \tilde{P}_{j}^{} )^{*}
        \tilde{D}_{j}^{} ( \tilde{R}_{j}^{} A \tilde{R}_{j}^{*} ) \tilde{D}_{j}^{}
        ( \tilde{I}_{j} - \tilde{P}_{j}^{} )^{}
        u_{j}
        = \lambda_{j} \tilde{C}_{j}^{} u_{j},
   \end{cases} 
   \quad j=1,\dots,J.
\end{equation}
We show now that it is enough to solve the local generalized eigenvalue problems
in the local \(A\)-harmonic space, namely we show that any solution
to~\eqref{eq:gevp_ras_proj} with non-trivial eigenvalue is such that
\(u_{j} \in \operatorname*{range} (I - \tilde{P}_{j}^{})\).

Let \(u_{j}\) be an eigenvector satisfying~\eqref{eq:gevp_ras_proj_CANeu} with
associated eigenvalue \(\lambda_{j} > \tau\).
Then multiplying~\eqref{eq:gevp_ras_proj_CANeu} with \(Q_{j}^{}\) we have:
\begin{align}\label{eq:gevp_tested_with_Qt}
    \left( Q_{j}^{} 
        (\tilde{I}_{j} - \tilde{P}_{j})^{*} 
        \tilde{D}_{j}^{} ( \tilde{R}_{j}^{} A \tilde{R}_{j}^{*} ) \tilde{D}_{j}^{}
        (\tilde{I}_{j} - \tilde{P}_{j})
    \right) u_{j}
    & = \lambda_{j} Q_{j}^{} \tilde{C}_{j}^{} \tilde{P}_{j}u_{j}
    + \lambda_{j} Q_{j}^{} \tilde{C}_{j}^{} (\tilde{I}_{j} - \tilde{P}_{j}^{}) u_{j}.
\end{align}
The left-hand-side vanishes since
\begin{align*}
    \left[Q_{j}^{} (\tilde{I}_{j} - \tilde{P}_{j})^{*}\right]^{*}
    & =
    (\tilde{I}_{j} - \tilde{P}_{j}) Q_{j}^{*}
    =
    (\tilde{I}_{j} - Q_{j}^{*} (Q_{j}^{} \tilde{C}_{j}^{} Q_{j}^{*})^{-1} Q_{j}^{} \tilde{C}_{j}^{})
    Q_{j}^{*}\\
    & =
    Q_{j}^{*}
    ( \tilde{I}_{j}
    -
    (Q_{j}^{} \tilde{C}_{j}^{} Q_{j}^{*})^{-1} (Q_{j}^{} \tilde{C}_{j}^{} Q_{j}^{*}))
    = 0.
\end{align*}
Besides, multiplying the above equation with \(\tilde{C}_{j}^{}\) we get
\begin{equation}
    \tilde{C}_{j}^{} (\tilde{I}_{j} - \tilde{P}_{j}) Q_{j}^{*} = 0,
    \quad\Leftrightarrow\quad
    Q_{j}^{} \tilde{C}_{j}^{} (\tilde{I}_{j}-\tilde{P}_{j}^{}) = 0,
\end{equation}
so that the second term in the right-hand-side
of~\eqref{eq:gevp_tested_with_Qt} also vanishes.
We finally obtain
\(
    Q_{j}^{} \tilde{C}_{j}^{} \tilde{P}_{j}^{} u_{j} = 0
\).
Since 
\(
    \operatorname*{ker} (Q_{j}^{} \tilde{C}_{j}^{}) \subset
    \operatorname*{ker} (\tilde{P}_{j}) =
    \operatorname*{range} (\tilde{I}_{j} - \tilde{P}_{j})
\),
we get \( \tilde{P}_{j}^{}u_{j} \in \operatorname{range} (\tilde{I}_{j} - \tilde{P}_{j}^{})\),
which implies \(\tilde{P}_{j}^{}u_{j}=0\) hence
\(u_{j} \in \operatorname*{range} (\tilde{I}_{j} - \tilde{P}_{j})\) which by
Lemma~\ref{lem:exactlocalsolvesHarmonic} is made of $A$-harmonic functions.

Overall, we have proved the following result.
\begin{lemma}
    \label{th:harmonicgev}
    The extended GenEO eigenvalue problem of Definition~\ref{def:extendedGenEO}
    can be reformulated directly as a problem on discrete \(A\)-harmonic
    vectors (\(\operatorname*{range} (\tilde{I}_{j} - \tilde{P}_{j}^{})\)):
    \begin{equation}\label{eq:gevp_ras_proj_harmonic}
       \begin{cases}
            \text{Find}\ 
            (\lambda_{j},u_{j}) \in \mathbb{R} \times
            \operatorname*{range} (\tilde{I}_{j} - \tilde{P}_{j}^{})
            \ \text{such that}:\\
            ( \tilde{I}_{j} - \tilde{P}_{j}^{} )^{*}
            \tilde{D}_{j}^{} ( \tilde{R}_{j}^{} A \tilde{R}_{j}^{*} ) \tilde{D}_{j}^{}
            ( \tilde{I}_{j} - \tilde{P}_{j}^{} )^{}
            u_{j}
            = \lambda_{j} \tilde{C}_{j}^{} u_{j},
       \end{cases} 
       \quad j=1,\dots,J.
    \end{equation}   
\end{lemma}
Therefore, the coarse space takes the following simpler form 
\begin{equation}
    Z := \bigoplus_{j=1}^{J}
    \tilde{R}_{j}^{*} \tilde{D}_{j}^{}
    \operatorname*{span}
    \left\{
        u_{j} \;|\;
        (\lambda_{j}, u_{j}) \ \text{solution of~\eqref{eq:gevp_ras_proj_harmonic} with }
        \lambda_{j} > \tau
    \right\}.
\end{equation}
The above local generalized eigenvalue
problems~\eqref{eq:gevp_ras_proj_harmonic} actually correspond to the ones
recently proposed in multiscale methods for generalized finite element methods,
see~\cite[Eq.~(3.13)]{Ma:2022:novel}.

\subsection{Additive Schwarz}\label{sec:as}

Our analysis also covers the case of the Additive Schwarz (AS)
method corresponding to the local solvers~\eqref{eq:as}, namely
\(S_{j}^{} := B_{j}^{-1}\).
In this case we obtain
\begin{equation}
    \tilde{S}_{j}
    = Q_{j}^{*} B_{j}^{-1} Q_{j},
    \qquad
    \tilde{L}_{j}
    = \tilde{D}_{j}^{} - Q_{j}^{*} B_{j}^{-1} Q_{j}^{} \tilde{R}_{j}^{} A \tilde{R}_{j}^{*},
    \qquad j=1,\dots,J,
\end{equation}
and the local generalized eigenvalue problems for \(j=1,\dots,J\) are then
\begin{equation}\label{eq:gevp_generic_as}
   \begin{cases}
        \text{Find}\ 
        (\lambda_{j},u_{j}) \in \mathbb{R} \times \operatorname*{range} \tilde{C}_{j}
        \ \text{such that}:\\
        P_{\tilde{C}_{j}}
        ( \tilde{D}_{j}^{} - Q_{j}^{*} B_{j}^{-1} Q_{j}^{} \tilde{R}_{j}^{} A \tilde{R}_{j}^{*} )^{*}
        ( \tilde{R}_{j}^{} C \tilde{R}_{j}^{*} )
        ( \tilde{D}_{j}^{} - Q_{j}^{*} B_{j}^{-1} Q_{j}^{} \tilde{R}_{j}^{} A \tilde{R}_{j}^{*} )^{}
        P_{\tilde{C}_{j}}
        u_{j}
        = \lambda_{j} \tilde{C}_{j} u_{j},
   \end{cases} 
\end{equation}
or equivalently with a more explicit formula:
\begin{equation}\label{eq:gevp_generic_as_smaller_decomposition}
   \begin{cases}
        \text{Find}\ 
        (\lambda_{j},u_{j}) \in \mathbb{R} \times \operatorname*{range} \tilde{C}_{j}
        \ \text{such that}:\\
        P_{\tilde{C}_{j}}
        \tilde{R}_{j}^{} 
        ( {D}_{j}^{} {R}_{j}^{} - B_{j}^{-1} {R}_{j}^{} A )^{*}
        ( {R}_{j}^{} C {R}_{j}^{*} )
        ( {D}_{j}^{} {R}_{j}^{} - B_{j}^{-1} {R}_{j}^{} A )^{}
        \tilde{R}_{j}^{*} 
        P_{\tilde{C}_{j}}
        u_{j}
        = \lambda_{j} \tilde{C}_{j} u_{j}.
   \end{cases} 
\end{equation}
The case of exact solvers corresponds as usual to taking
\(B_{j} = R_{j}^{} A R_{j}^{*}\).
The connection with the GenEO method is detailed in \S~\ref{sec:geneo} below.

\subsection{Symmetrized Optimized Restricted Additive Schwarz}\label{sec:soras}

Our analysis also covers the case of the Symmetrized Optimized Restricted
Additive Schwarz (SORAS) method corresponding to the local
solvers~\eqref{eq:soras}, namely 
\(S_{j}^{} := D_{j}^{} B_{j}^{-1} D_{j}^{}\).
In this case we obtain
\begin{equation}
    \tilde{S}_{j}
    = \tilde{D}_{j}^{} Q_{j}^{*} B_{j}^{-1} Q_{j} \tilde{D}_{j}^{},
    \qquad
    \tilde{L}_{j}
    = \tilde{D}_{j}^{} (\tilde{I}_{j} - Q_{j}^{*} B_{j}^{-1} Q_{j}^{} \tilde{D}_{j}^{} \tilde{R}_{j}^{} A \tilde{R}_{j}^{*}),
    \qquad j=1,\dots,J,
\end{equation}
and the local generalized eigenvalue problems for \(j=1,\dots,J\) are then
\begin{equation}\label{eq:gevp_generic_soras}
   \begin{cases}
        \text{Find}\ 
        (\lambda_{j},u_{j}) \in \mathbb{R} \times \operatorname*{range} \tilde{C}_{j}
        \ \text{such that}:\\
        P_{\tilde{C}_{j}}
        ( \tilde{I}_{j} - Q_{j}^{*} B_{j}^{-1} Q_{j}^{} \tilde{D}_{j}^{} \tilde{R}_{j}^{} A \tilde{R}_{j}^{*} )^{*}
        \tilde{D}_{j}^{}
        ( \tilde{R}_{j}^{} C \tilde{R}_{j}^{*} )
        \tilde{D}_{j}^{}
        ( \tilde{I}_{j} - Q_{j}^{*} B_{j}^{-1} Q_{j}^{} \tilde{D}_{j}^{} \tilde{R}_{j}^{} A \tilde{R}_{j}^{*} )^{}
        P_{\tilde{C}_{j}}
        u_{j}\\
        \qquad\qquad= \lambda_{j} \tilde{C}_{j} u_{j},
   \end{cases} 
\end{equation}
or equivalently with a more explicit formula:
\begin{equation}\label{eq:gevp_generic_soras_smaller_decomposition}
   \begin{cases}
        \text{Find}\ 
        (\lambda_{j},u_{j}) \in \mathbb{R} \times \operatorname*{range} \tilde{C}_{j}
        \ \text{such that}:\\
        P_{\tilde{C}_{j}}
        \tilde{R}_{j}^{} 
        ( R_{j}^{} - B_{j}^{-1} {D}_{j}^{} {R}_{j}^{} A )^{*}
        {D}_{j}^{}
        ( {R}_{j}^{} C {R}_{j}^{*} )
        {D}_{j}^{}
        ( R_{j}^{} - B_{j}^{-1} {D}_{j}^{} {R}_{j}^{} A )^{}
        \tilde{R}_{j}^{*} 
        P_{\tilde{C}_{j}}
        u_{j}
        = \lambda_{j} \tilde{C}_{j} u_{j}.
   \end{cases} 
\end{equation}
The case of exact solvers corresponds as usual to taking
\(B_{j} = R_{j}^{} A R_{j}^{*}\).

\begin{remark}
    For this method, it may not be necessary to extend the decomposition as in
    Section~\ref{sec:extension} if the partition of unity vanishes on the
    boundary of the subdomains.
    Indeed, then \(D_j R_j A= D_j R_j A R_j^{*} R_j\) so that
    in~\eqref{eq:one-level-prop-error} the most right term can be rewritten as
    follows:
    \begin{equation}
      S_j R_j A = D_j B_j^{-1} D_j R_j A = D_j B_j^{-1} D_j R_j A R_j^{*} R_j.
    \end{equation}
    Then, the propagation error operator (see~\eqref{eq:one-level-prop-error})
    of the one-level method reads:
    \begin{align}
        I - \sum_{j=1}^{J} M_{j}^{-1} A
        & = I - \sum_{j=1}^{J} R_{j}^{*} S_{j} R_{j}^{} A
        = \sum_{j=1}^{J} R_{j}^{*} (D_{j}^{} - D_j B_j^{-1} D_j R_j A R_j^{*}) R_j,
    \end{align}
    is decomposed as a sum of local contributions without the need for an
    extension of the subdomains. 
\end{remark}

The connection with the SORAS-GenEO method is detailed in \S~\ref{sec:geneo2-soras} below.

\section{Comparison with existing coarse spaces}\label{sec:comparision_with_existing_CS}

We compare the coarse spaces obtained from the above analysis with alternative
constructions already available in the literature for SPD problems.
In this section we take \(C=A\) to be SPD.

\subsection{Standard GenEO coarse space for Additive Schwarz}\label{sec:geneo}

We can compare with the standard two-level GenEO preconditioner for the
Additive Schwarz (AS) method with exact solvers as given
in~\cite{Dolean:2015:IDDSiam} (which is slightly different from but very
similar to the one in~\cite{Spillane:2014:ASC}).
Since the analysis in this reference is based on the fictitious lemma, the
estimates are finer but restricted to symmetric preconditioners.
The algorithm considered there corresponds to defining the local solvers as 
\begin{equation}
    S_{j}^{} := (R_{j}^{} A R_{j}^{*})^{-1},
    \qquad j=0,\dots,J.
\end{equation}
Both the local preconditioners \(M_{j}^{-1}\), \(j=1,\dots,J\), for the first
level and the coarse space preconditioner \(M_{0}^{-1}\) for the second level
are applied in an additive fashion, namely using our notations the full
two-level preconditioner is~\cite[Eq.~(7.6)]{Dolean:2015:IDDSiam}
\begin{equation}
    M_{\text{GenEO}}^{-1} := \sum_{j=0}^{J} M_{j}^{-1},
    \qquad M_{j} := R_{j}^{*} S_{j}^{} R_{j}^{}
    = R_{j}^{*} (R_{j}^{} A R_{j}^{*})^{-1} R_{j}^{},
    \quad j=0,\dots,J.
\end{equation}
The local generalized eigenvalue problems are derived from different
considerations than the ones above.
They are given in~\cite[Def.~7.14]{Dolean:2015:IDDSiam}, 
and correspond to
\begin{equation}\label{eq:gevp_geneo}
   \begin{cases}
        \text{Find}\ 
        (\lambda_{j},u_{j}) \in \mathbb{R} \times \operatorname*{range} {A}_{j}^{N}
        \ \text{such that}:\\
        P_{{A}_{j}^{N}}
        D_{j}^{} R_{j}^{} A R_{j}^{*} D_{j}^{}
        P_{{A}_{j}^{N}}
        u_{j}
        = \lambda_{j} {A}_{j}^{N} u_{j},
   \end{cases} 
   \qquad j=1,\dots,J,
\end{equation}
where \({A}_{j}^{N}\) are the so-called local Neumann matrices
associated to the decomposition \(R_{j}^{}\)
and \(P_{{A}_{j}^{N}}\) is the orthogonal projection on
\(\operatorname*{range} {A}_{j}^{N}\).
The associated coarse space is defined in~\cite[Def.~7.16]{Dolean:2015:IDDSiam}
and corresponds to
\begin{equation}
    Z_{\text{GenEO}} := \bigoplus_{j=1}^{J}
    R_{j}^{*} D_{j}^{} 
    \left[
    \ker {A}_{j}^{N} \bigoplus \operatorname*{span}
    \left\{
        u_{j} \;|\;
        (\lambda_{j}, u_{j}) \ \text{solution of~\eqref{eq:gevp_geneo} with }
        \lambda_{j} > \tau
    \right\}
    \right].
\end{equation}
To make a better comparison with Section~\ref{sec:ras_harmonic} let us assume
also here that \(A_{j}^{N}\) is invertible so that the local generalized
eigenvalue problems for the GenEO method are
\begin{equation}\label{eq:gevp_geneo_ANnoker}
   \begin{cases}
        \text{Find}\ 
        (\lambda_{j},u_{j}) \in \mathbb{R} \times
        \mathbb{C}^{\#{\mathcal{N}}_{j}}
        \ \text{such that}:\\
        D_{j}^{} R_{j}^{} A R_{j}^{*} D_{j}^{}
        u_{j}
        = \lambda_{j} {A}_{j}^{N} u_{j},
   \end{cases} 
   \qquad j=1,\dots,J.
\end{equation}
The associated GenEO coarse space is
\begin{equation}
    Z_{\text{GenEO}} := \bigoplus_{j=1}^{J}
    R_{j}^{*} D_{j}^{} 
    \left[
    \operatorname*{span}
    \left\{
        u_{j} \;|\;
        (\lambda_{j}, u_{j}) \ \text{solution of~\eqref{eq:gevp_geneo} with }
        \lambda_{j} > \tau
    \right\}
    \right].
\end{equation}
It is remarkable that the difference between~\eqref{eq:gevp_ras_proj_harmonic}
and~\eqref{eq:gevp_geneo_ANnoker} essentially lies in the space in which one
looks for the eigenvectors.
In~\eqref{eq:gevp_ras_proj_harmonic}, the space is made of discrete
\(A\)-harmonic vectors whereas in~\eqref{eq:gevp_geneo_ANnoker} the vectors are
\(A\)-harmonic only for degrees of freedom for which the partition of unity is
equal to one. 
The other difference is that~\eqref{eq:gevp_ras_proj_harmonic} requires an
extended decomposition.

\subsection{GenEO coarse space for SORAS}\label{sec:geneo2-soras}

The literature based on the fictitious space lemma proposes for the SORAS
variant the construction of coarse spaces using the solutions of two local
generalized eigenproblems,
see~\cite{Haferssas:ASM:2017} and~\cite[Sec.~7.7]{Dolean:2015:IDDSiam}.
More precisely, the two local generalized eigenproblems considered are
\begin{equation}\label{eq:gevp_geneo_soras_1}
   \begin{cases}
        \text{Find}\ 
        (\lambda_{j},u_{j}) \in \mathbb{R} \times
        \mathbb{C}^{\#\tilde{\mathcal{N}}_{j}}
        \ \text{such that}:\\
        D_{j}^{} R_{j}^{} A R_{j}^{*} D_{j}^{}
        u_{j}
        = \lambda_{j} {B}_{j}^{} u_{j},
   \end{cases} 
   \qquad j=1,\dots,J,
\end{equation}
\begin{equation}\label{eq:gevp_geneo_soras_2}
   \begin{cases}
        \text{Find}\ 
        (\lambda_{j},u_{j}) \in \mathbb{R} \times 
        \mathbb{C}^{\#\tilde{\mathcal{N}}_{j}}
        \ \text{such that}:\\
        A_{j}^{N}
        u_{j}
        = \lambda_{j} {B}_{j}^{} u_{j},
   \end{cases} 
   \qquad j=1,\dots,J,
\end{equation}
where here the local matrices \(B_{j}\) are SPD matrices which typically come
from the discretization of boundary value local problems using optimized
transmission conditions (e.g.\ Robin boundary conditions).
For two thresholds \(\tau>0\) and \(\gamma>0\),
the associated coarse space is
\(Z = Z_{\tau} \oplus Z_{\gamma}\)
where
\begin{align}
    & Z_{\tau} := 
    \bigoplus_{j=1}^{J}
    R_{j}^{*} D_{j}^{} 
    \left[
    \operatorname*{span}
    \left\{
        u_{j} \;|\;
        (\lambda_{j}, u_{j}) \ \text{solution of~\eqref{eq:gevp_geneo_soras_1} with }
        \lambda_{j} > \tau
    \right\}
    \right],\\
    \label{eq:gammageneo2}
    & Z_{\gamma} := 
    \bigoplus_{j=1}^{J}
    R_{j}^{*} D_{j}^{} 
    \left[
    \operatorname*{span}
    \left\{
        u_{j} \;|\;
        (\lambda_{j}, u_{j}) \ \text{solution of~\eqref{eq:gevp_geneo_soras_2} with }
        \lambda_{j} < \gamma
    \right\}
    \right].
\end{align}
As a result these coarse spaces differ fundamentally from the one we propose
above which are based on the local generalized
eigenproblems~\eqref{eq:gevp_generic_soras}.

\section{Numerical experiments}

All numerical results presented below are obtained using the FreeFEM
software~\cite{Hecht:2012:NDF}.

\subsection{Non-symmetric and symmetric preconditioners for SPD problems}

We consider first the following two-dimensional model problem
\begin{equation}
    \begin{cases}
        -\operatorname{div} \nu \operatorname{grad} u + \eta u = f,
        &\qquad\text{in}\ \Omega,\\
        \partial_{\mathbf{n}}u + u = 0, 
        &\qquad\text{in}\ \Gamma_{R},\\
        \partial_{\mathbf{n}}u = 0, 
        &\qquad\text{in}\ \Gamma_{N},
    \end{cases}
\end{equation}
where the domain is a square
\(\Omega := (0,\ell)\times(0,\ell)\)
with \(\ell\) the square root of the number of subdomains, \(\ell:=\sqrt{J}\), 
and its boundary is given by
\(\Gamma_{R} := (0,\ell)\times\{0\}\) and
\(\Gamma_{N} := (0,\ell)\times\{\ell\} \cup \{0,\ell\}\times(0,\ell)\).
The outward normal vector to \(\partial\Omega\) is denoted \(\mathbf{n}\).

The source term is \(f=1\), the coefficient \(\eta=10^{-8}\)
(to ensure well-posedness of all local subdomain problems),
and the coefficient \(\nu\) is either constant equal to \(1\)
(we refer to this situation as the homogeneous coefficient case)
or piecewise constant with high-contrast
(we refer to this situation as the heterogeneous coefficient case):
\(\nu=1+10^{5}\) on \((2\ell/10,4\ell/10)\times(0,1)\);
\(\nu=1+10^{4}\) on \((6\ell/10,8\ell/10)\times(0,1)\);
and \(\nu=1\) otherwise.

By standard arguments the bilinear form associated to the above model problem
is symmetric and coercive (positive definite).
After discretization using conformal \(\mathbb{P}_{1}\) Lagrange finite
elements on a regular triangular mesh, the associated matrix of the linear
system \(A\) is then SPD.

\subsubsection{Exact solvers}

To investigate the efficiency of our two-level coarse space construction on various
one-level methods we consider three variants of additive Schwarz methods:
\begin{enumerate}
    \item Restricted Additive Schwarz (RAS), see \S~\ref{sec:ras};
    \item Additive Schwarz (AS), see \S~\ref{sec:as};
    \item Symmetrized Optimized Restricted Additive Schwarz (SORAS), see \S~\ref{sec:soras}.
        The local matrices involved in the local solvers are obtained from the
        bilinear form of the original problem by adding a term stemming from
        the artificial Robin boundary condition \(\partial_{\mathbf{n}}u +
        \nu/\sqrt{h}\; u=0\) on the artificial boundary, where \(h\) is the
        local mesh size, see~\cite{Gander:2006:OSM} for the connection with
        Optimized or order 0 interface conditions.
\end{enumerate}
We report numerical results for three coarse space constructions:
\begin{enumerate}
    \item our extended GenEO coarse space construction given in \S~\ref{sec:new_cs}
        (labelled `extended' in the legends of the forthcoming plots).
        Notice that this coarse space \emph{does} depend on the one-level
        preconditioner.
        The extended decomposition described in \S~\ref{sec:extension} is
        achieved by extending the overlap by one additional layer of cells.
        Since \(A\) is SPD, we set \(C=A\) and we use Robin matrices for the
        local matrices \(\tilde{C}_{j}\), i.e.~matrices
        assembled from the bilinear form of the original problem with an
        artificial Robin boundary condition
        \(\partial_{\mathbf{n}}u+10^{-4}u=0\) on the artificial boundary.
        This particular choice of local matrices is made to numerically
        stabilize the eigensolvers.
        Since the Robin parameter is small, this matrix is numerically similar
        to the Neumann matrix (which arises in the GenEO eigenproblems).
        Since the local matrices \(\tilde{C}_{j}\) are invertible, the
        projection operators \(P_{\tilde{C}_{j}}\) in the local generalized
        eigenproblems become the identity.
    \item a standard GenEO coarse space as detailed in \S~\ref{sec:geneo}
        (labelled `GenEO' in the legends of the forthcoming plots).
        This coarse space does \emph{not} depend on the one-level
        preconditioner, but its construction is supported by theoretical
        analysis for one-level AS only.
        Thus, when used together with the RAS or SORAS one-level preconditioners,
        we use the coarse space constructed for the AS preconditioner.
        Since the local matrices \(A^{N}_{j}\) are invertible, the
        projection operators \(P_{A^{N}_{j}}\) in the local generalized
        eigenproblems become the identity.
    \item a GenEO-2 coarse space, for one-level SORAS only, as detailed in
        \S~\ref{sec:geneo2-soras} (labelled `GenEO-2' in the legends of the
        forthcoming plots).
\end{enumerate}
The efficiency of these three coarse spaces is compared with the pure one-level
method without coarse space (in which case the legend entry only features the
type of the one-level method).

In all three cases we consider a two-level preconditioner \(M^{-1}\)
constructed by adding a multiplicative global coarse correction to the
first-level residual propagation operator, see~\eqref{eq:full_preconditioner}.
The thresholds for selecting the admissible eigenvalues in the construction of
the coarse spaces are set to \(\tau=10\) for the extended GenEO coarse space
and the GenEO coarse space, and \(\gamma=1/10\) for GenEO-2,
see~\eqref{eq:gammageneo2}.
The eigenvalue problems are solved using SLEPc.

We report weak scaling results, with around \(N_{j} \approx 1600\) degrees of
freedom in each subdomain.
The overlapping decompositions are obtained thanks to the automatic graph
partitioner Metis.
There are about \(4\) layers of cells in the overlap and the partition of unity
is such that it vanishes inside the subdomain before reaching the subdomain
boundary (so that also its derivatives vanish at the boundary).
We report below convergence results of the GMRES algorithm applied to the
original linear system~\eqref{eq:syslin} with right-preconditioning by the
one-level or two-level preconditioners.
We also report results obtained with GMRES without preconditioning, which are
labelled as `none' in the legends of the forthcoming plots.
The convergence tolerance on the relative residual is set to \(10^{-8}\) and
\(10^{-6}\) for the homogeneous and heterogeneous medium respectively
and the maximum number of iterations if the
tolerance is not reached is fixed to \(200\) iterations.

A convergence history for \(J=50\) subdomains, the iteration count of the
GMRES algorithm to reach the given tolerance and the relative coarse space size
(i.e.\ the total size of the coarse space divided by the number of subdomains)
are reported in Figure~\ref{fig:reglaplacian_homogeneous_raspart} for the
homogeneous medium and in Figure~\ref{fig:reglaplacian_heterogeneous_raspart}
for the heterogeneous medium.
The first observation is that GMRES without preconditioning does not even get
one digit of accuracy in 200 iterations, which is the motivation to use domain
decomposition methods in the first place.
One the other hand, all one-level methods without coarse space exhibit some
convergence for small enough problems.
However, the convergence deteriorates as the size of the problem increases, the
number of iterations growing proportionally to the square root of the number of
subdomains which is also the number of subdomains in one direction since we are
in two-dimensions.
For large enough problems, the convergence of the algorithm can stall before
reaching convergence (which explains that sometimes no iteration count is
reported in the plots).
This effect is worsened in the case of heterogeneous medium.
This is the motivation for two-level methods.

In contrast, we observe overall good scalability properties of all two-level
methods tested, the lowest iteration counts being obtained when the underlying
one-level preconditioner is RAS.
We notice a slightly reduced iteration count for one-level preconditioners RAS
and AS with the GenEO coarse space compared to the extended GenEO coarse space,
but the size of the extended GenEO coarse space is comparatively smaller.
Note that in this case, our extended GenEO coarse space is very
similar to~\cite{Ma:2024:TLR} where they also reported smaller coarse space
sizes.
We recall that to the best of our knowledge no general theory was available for
the GenEO coarse space together with the non-symmetric one-level preconditioner
RAS.
For one-level preconditioner SORAS we notice a slightly increased iteration
count when using the GenEO-2 coarse space compared to the GenEO coarse space,
despite the latter having a larger coarse space, but we note that a complete
analysis theory is available only for the former, 
see~\cite{Haferssas:ASM:2017} and~\cite[Sec.~7.7]{Dolean:2015:IDDSiam}.

\begin{figure}[p]
    \centering
        {\includegraphics[height=0.175\textheight]{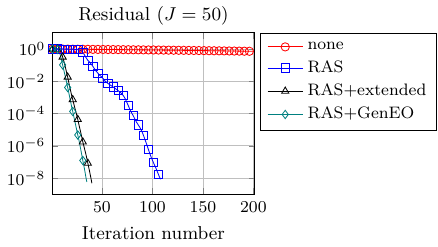}}
        \hfill
        {\includegraphics[height=0.175\textheight]{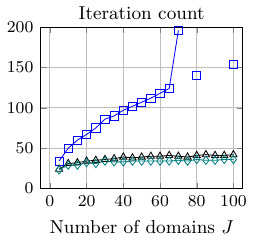}}
        {\includegraphics[height=0.175\textheight]{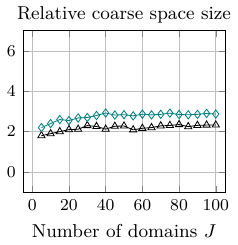}}
        {\includegraphics[height=0.175\textheight]{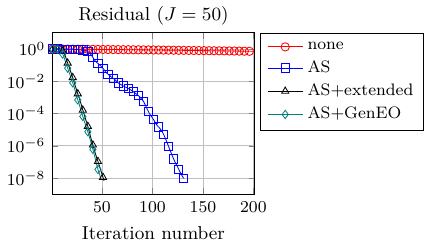}}
        \hfill
        {\includegraphics[height=0.175\textheight]{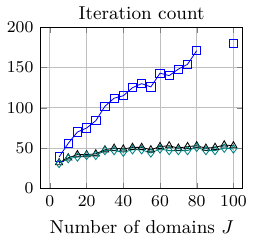}}
        {\includegraphics[height=0.175\textheight]{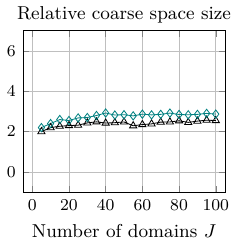}}
        {\includegraphics[height=0.175\textheight]{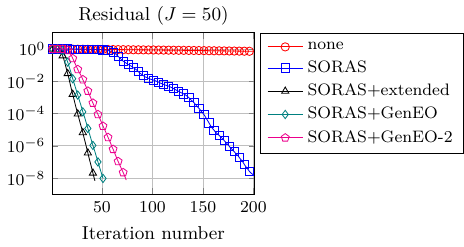}}
        \hfill
        {\includegraphics[height=0.175\textheight]{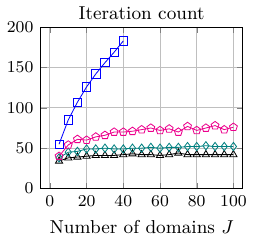}}
        {\includegraphics[height=0.175\textheight]{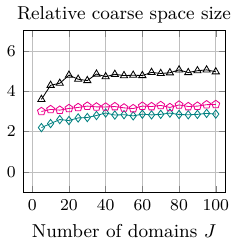}}
    \caption{
        Regularized Laplacian on homogeneous medium
        for one-level
        RAS (top), AS (middle) and SORAS (bottom).
    }\label{fig:reglaplacian_homogeneous_raspart}
\end{figure}

\begin{figure}[p]
    \centering
        {\includegraphics[height=0.175\textheight]{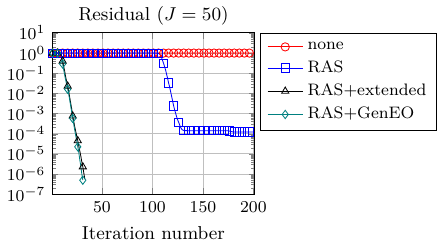}}
        \hfill
        {\includegraphics[height=0.175\textheight]{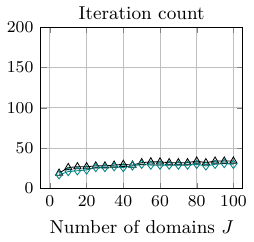}}
        {\includegraphics[height=0.175\textheight]{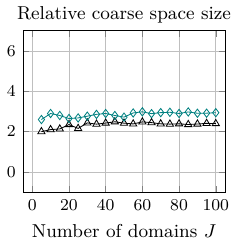}}
        {\includegraphics[height=0.175\textheight]{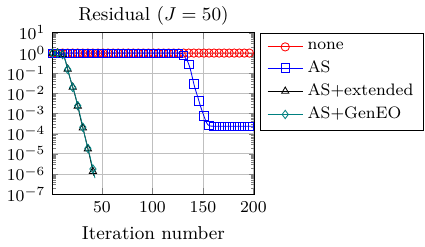}}
        \hfill
        {\includegraphics[height=0.175\textheight]{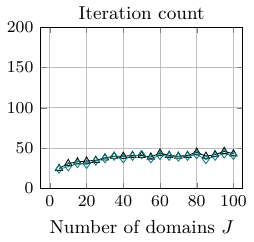}}
        {\includegraphics[height=0.175\textheight]{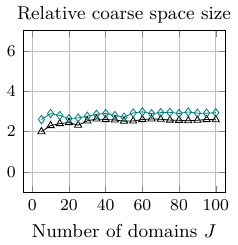}}
        {\includegraphics[height=0.175\textheight]{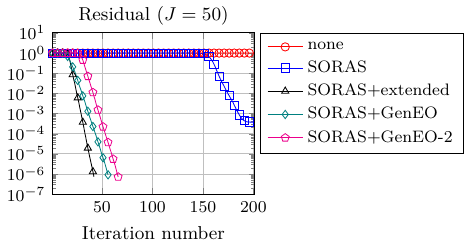}}
        \hfill
        {\includegraphics[height=0.175\textheight]{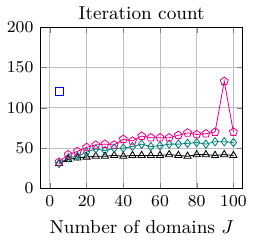}}
        {\includegraphics[height=0.175\textheight]{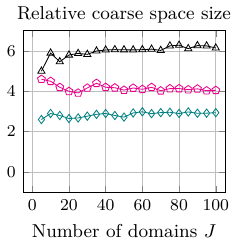}}
    \caption{
        Regularized Laplacian on heterogeneous medium
        for one-level
        RAS (top), AS (middle) and SORAS (bottom).
    }\label{fig:reglaplacian_heterogeneous_raspart}
\end{figure}

\subsubsection{Inexact solvers}

In order to speed up the one-level RAS method, local solves can be approximated
by a forward-backward substitution with an incomplete Cholesky factorization
(ICC).
This one level preconditioner is a common choice for a lightweight parallel
solver when using the PETSc library.

We repeat the previous numerical experiments using the incomplete Cholesky
factorization for the one-level RAS method.
The construction of the extended GenEO coarse space is affected by this change, and the
associated theory developed above covers this particular case, see
\S~\ref{sec:ras_inexact}.
We compare with the standard GenEO coarse space of the previous experiment,
which is not modified.
We note that the available GenEO coarse space
analysis allowing for inexact solvers~\cite[Sec.~4.5.3]{Spillane2025} (with
numerical results~\cite[Fig.~4]{Spillane2025}) features a coarse
space construction with two eigenproblems that is different from the one we
test here and has theoretical guarantees valid when used with a symmetric
preconditioner, which is not the case here.

A convergence history for \(J=50\) subdomains, the iteration count of the
GMRES algorithm to reach the given tolerance and the relative coarse space size
(i.e.\ the total size of the coarse space divided by the number of subdomains)
are reported in Figure~\ref{fig:reglaplacian_inexact_raspart}.
Again we observe that two-level preconditioning seems mandatory to obtain
robust GMRES convergence in a reasonable number of iterations.
The iteration count for both coarse spaces (extended GenEO and GenEO) is very
similar.
This is achieved with a slightly smaller coarse space size for the extended
GenEO coarse space compared to GenEO.
Compared to the case of exact solvers, the iteration count is significantly
increased.
To better understand this increase, we report that a GMRES solver on a single
domain of size similar to a single subdomain converges to the target tolerance
in \(63\) and \(89\) iterations for the homogeneous and heterogeneous
respectively.
If the second level cannot correct this inevitable increase, the preconditioner
is now scalable.
Besides, the use of approximate solvers has reduced the computational cost of
each iteration, as well as the memory load, compared to direct solvers.

\begin{figure}[p]
    \centering
        {\includegraphics[height=0.175\textheight]{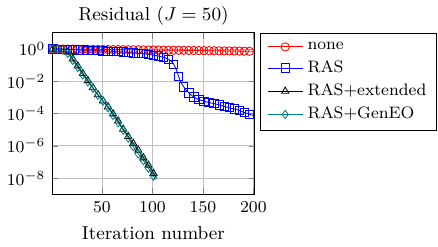}}
        {\includegraphics[height=0.175\textheight]{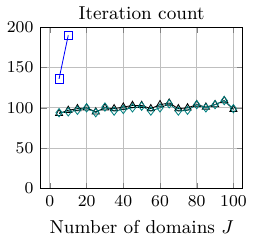}}
        {\includegraphics[height=0.175\textheight]{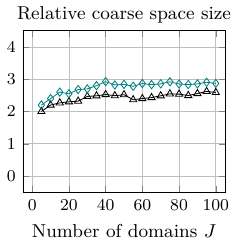}}
        {\includegraphics[height=0.175\textheight]{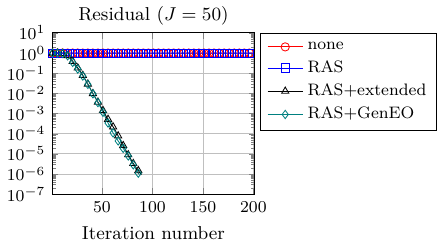}}
        {\includegraphics[height=0.175\textheight]{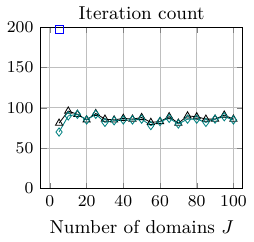}}
        {\includegraphics[height=0.175\textheight]{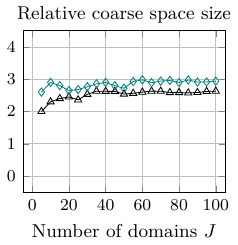}}
    \caption{
        Regularized Laplacian on homogeneous (top) and heterogeneous (bottom)
        medium, incomplete Cholesky factorization (ICC) of local matrices used
        in the one-level preconditioner and in the construction of the extended
        GenEO coarse space.
    }\label{fig:reglaplacian_inexact_raspart}
\end{figure}

\subsection{Non-symmetric problem}

Using the notations of the previous subsection,
we consider now the following two-dimensional non-symmetric problem
\begin{equation}
    \begin{cases}
        -\operatorname{div} \nu \operatorname{grad} u
        +\operatorname{div} (\mathbf{b}u)
        + \eta u = f,
        &\qquad\text{in}\ \Omega,\\
        \partial_{\mathbf{n}}u + u = 0, 
        &\qquad\text{in}\ \Gamma_{R},\\
        \partial_{\mathbf{n}}u = 0, 
        &\qquad\text{in}\ \Gamma_{N},
    \end{cases}
\end{equation}
where \(\mathbf{b} \in \mathbb{R}^{2}\).
We consider two different choices for \(\mathbf{b}\),
\begin{itemize}
    \item \(\mathbf{b} = [1, 0]^{T}\), referred to as the constant advection
        test case;
    \item \(\mathbf{b} = [(2y-1)\pi, (2x-1)\pi]^{T}\), a divergence free
        rotating flow, referred as the varying advection test case.
\end{itemize}
The presence of the first order term in the differential operator makes the
problem non-symmetric and the same discretization method yields a non-symmetric
matrix \(A\).

The discrete variational formulation is stabilized using the Streamline Upwind
Petrov--Galerkin method, see e.g.~\cite{Bonazzoli:2020:ASD}
and~\cite[\S~11.8.6]{Quarteroni2009} for the details.

In this case, the matrix \(A\) is not SPD and we construct \(C\) from the
symmetric part of \(A\).
In order to numerically stabilize the eigensolvers, the local matrices
\(\tilde{C}_{j}\) are again taken to be the Robin matrices formed from the
symmetric part of the bilinear form with an artificial Robin boundary condition
\(\partial_{\mathbf{n}}u+10^{-4}u=0\) on the artificial boundary.

For this non-symmetric problems, we only consider the ORAS preconditioner for
the one-level method and the extended GenEO coarse space for the second-level.
The local matrices involved in the local solvers are obtained from the skew symmetric
bilinear form of the original problem by adding a term stemming from
the artificial Robin boundary condition \(\partial_{\mathbf{n}}u +
\sqrt{\mathbf{b}\cdot\mathbf{n}+4\nu}/(2\sqrt{h})\; u=0\) on the artificial
boundary, where \(h\) is the local mesh size.
Given the increased difficulty of these numerical experiments, we vary the
threshold in the selection of the eigenvalues \(\tau \in \{10,1\}\) in the
construction of the coarse space.
All other parameters are retained from the previous results.

A convergence history for \(J=50\) subdomains, the iteration count of the
GMRES algorithm to reach the given tolerance and the relative coarse space size
(i.e.\ the total size of the coarse space divided by the number of subdomains)
are reported in Figure~\ref{fig:convdiff_homogeneous} for \(\nu=1\),
in Figure~\ref{fig:convdiff_heterogeneous} for a medium with a large contrast in \(\nu\)
and in Figure~\ref{fig:convdiff_homogeneous_smallnu} for \(\nu=10^{-3}\).
For this non-symmetric problem, we observe again a clear gain in robustness and
scalability using the second-level method.
However, we notice that it may happen that the one-level method converges in marginally
fewer iterations than the two-level one which is not in contradiction with the theory.
Decreasing the threshold \(\tau\) involves significantly larger coarse spaces
but implies a decrease on the iteration count.

\begin{figure}[p]
    \centering
        {\includegraphics[height=0.175\textheight]{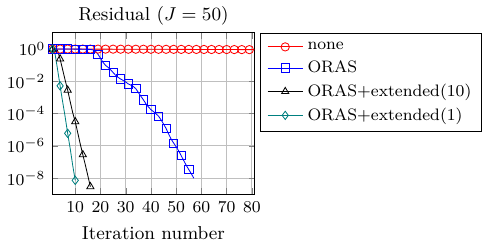}}
        {\includegraphics[height=0.175\textheight]{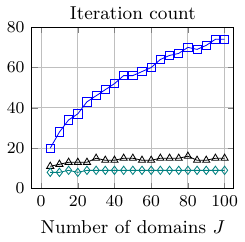}}
        {\includegraphics[height=0.175\textheight]{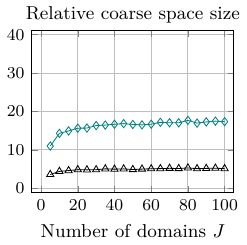}}
        {\includegraphics[height=0.175\textheight]{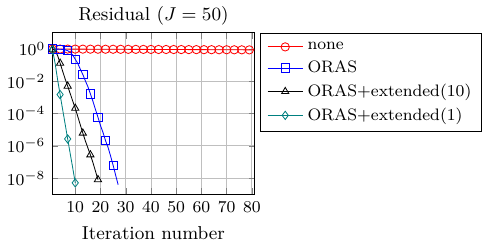}}
        {\includegraphics[height=0.175\textheight]{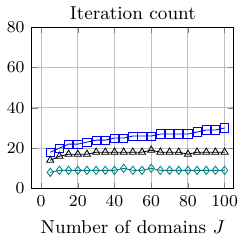}}
        {\includegraphics[height=0.175\textheight]{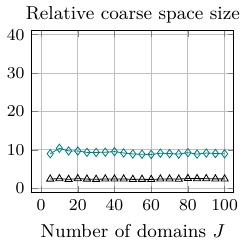}}
    \caption{
        Convection-diffusion with constant \(\nu=1\); for constant advection (top) and varying advection (bottom).
        The number in the parenthesis is the threshold \(\tau\) for the eigenvalue selection.
    }\label{fig:convdiff_homogeneous}
\end{figure}

\begin{figure}[p]
    \centering
        {\includegraphics[height=0.17\textheight]{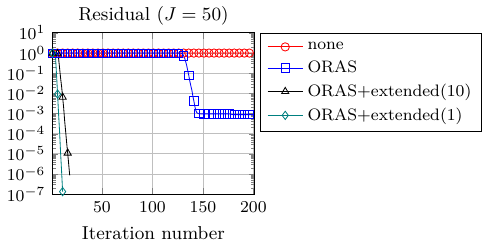}}
        {\includegraphics[height=0.17\textheight]{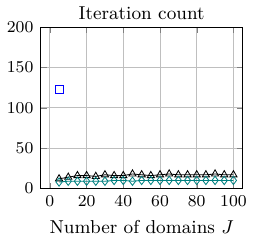}}
        {\includegraphics[height=0.17\textheight]{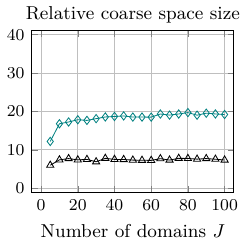}}
        {\includegraphics[height=0.17\textheight]{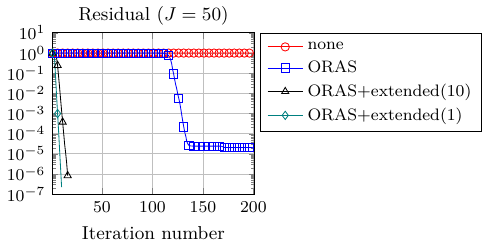}}
        {\includegraphics[height=0.17\textheight]{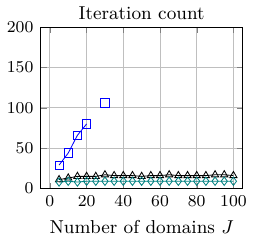}}
        {\includegraphics[height=0.17\textheight]{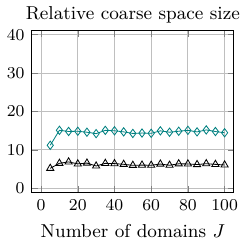}}
    \caption{
        Convection-diffusion with large variations of \(\nu\):
        (\(\nu=1+10^{5}\) on \((2\ell/10,4\ell/10)\times(0,1)\)
        \(\nu=1+10^{4}\) on \((6\ell/10,8\ell/10)\times(0,1)\)
        and \(\nu=1\) otherwise);
        for constant advection (top) and varying advection (bottom).
        The number in the parenthesis is the threshold \(\tau\) for the eigenvalue selection.
    }\label{fig:convdiff_heterogeneous}
\end{figure}

\begin{figure}[p]
    \centering
        {\includegraphics[height=0.175\textheight]{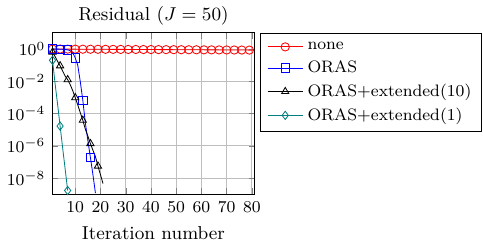}}
        {\includegraphics[height=0.175\textheight]{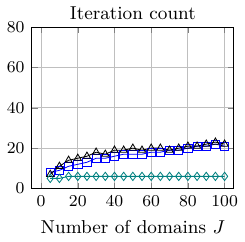}}
        {\includegraphics[height=0.175\textheight]{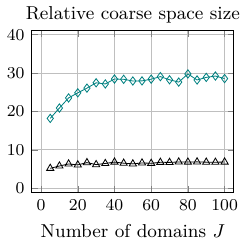}}
        {\includegraphics[height=0.175\textheight]{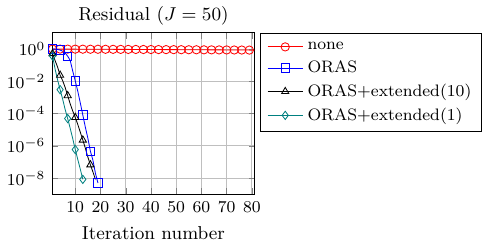}}
        {\includegraphics[height=0.175\textheight]{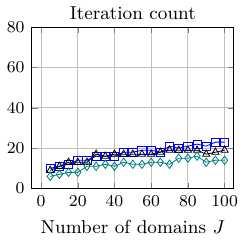}}
        {\includegraphics[height=0.175\textheight]{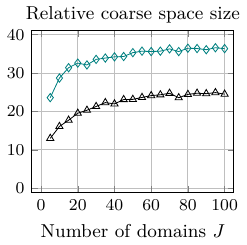}}
    \caption{
        Convection-diffusion with constant \(\nu=10^{-3}\); for constant
        advection (top) and varying advection (bottom).
        The number in the parenthesis is the threshold \(\tau\) for the
        eigenvalue selection.
    }\label{fig:convdiff_homogeneous_smallnu}
\end{figure}

\section{Conclusion}

We have performed a general analysis of adaptive coarse spaces that applies to
symmetric and non-symmetric problems, to symmetric preconditioners such the
additive Schwarz (AS) method and the non-symmetric preconditioner restricted
additive Schwarz (RAS), as well as to exact or inexact subdomain solves.
This led us to the design of an extended GenEO coarse space in Definition~\ref{def:newCS}.
The coarse space is built by solving generalized eigenvalues in the subdomains
and applying a well-chosen operator to the selected eigenvectors.

Note that the standard theory of adaptive coarse space for the AS method makes
use of the stable decomposition concept~\cite{Xu:1989:TMM} or of the Fictitious
Space Lemma~\cite{Nepomnyaschikh:1991:MTT,Griebel:1995:ATA} valid in the
framework of symmetric positive operators. It yields sharp spectral estimates
on the preconditioned operator. Here we do not use these theories and our
results are restricted to estimates of the norm of the error propagation
operator. However, they hold for non-symmetric operators and/or preconditioners.

\section*{Acknowledgments}
The authors would like to thank Pierre-Henri Tournier and Pierre Jolivet for
their help with the FreeFEM implementation, Marcella Bonazzoli for her FreeFEM
script on the convection-diffusion problem, Pierre Marchand and Nicole Spillane
for useful discussions.

\printbibliography[heading=bibintoc,title={References}]

\end{document}